\documentclass[
final
, nomarks
]{dmtcs-episciences}

\usepackage{latexsym}
\usepackage{amsfonts}
\usepackage{amsthm}
\usepackage{amsmath}
\usepackage{amssymb}
\usepackage{amscd}
\usepackage[hang,flushmargin]{footmisc}
\usepackage[utf8]{inputenc}
\usepackage{t1enc}
\usepackage{tabu}
\usepackage[mathscr]{eucal}
\usepackage{graphicx}
\usepackage{graphics}
\usepackage{pict2e}
\numberwithin{equation}{section}
\usepackage{epstopdf} 
\usepackage{float}
\usepackage{tikz}
\usetikzlibrary{fadings}
\usepackage{makecell}
\usepackage{subfigure}
\usepackage[round]{natbib}

\newtheorem{Th}{Theorem}[section]
\newtheorem{Lemma}[Th]{Lemma}

\newtheorem{Prop}[Th]{Proposition}

\theoremstyle{definition}
\newtheorem{Def}[Th]{Definition}

\newtheorem{?}[Th]{Problem}

\DeclareMathOperator{\DB}{DB} 
\DeclareMathOperator{\Dec}{Dec} 
\DeclareMathOperator{\tl}{tl}
\DeclareMathOperator{\Inc}{Inc}

\author{Hanna Mularczyk\thanks{Work supported by NSF/DMS grant 1659047 and NSA grant H98230-18-1-0010.}}

\title{Lattice Paths and Pattern-Avoiding Uniquely Sorted Permutations}

\affiliation{Harvard University, Cambridge MA}

\keywords{uniquely sorted permutation, stack-sorting, lattice path, valid hook configuration, permutation pattern.}

\received{2020-05-20}

\revised{2021-04-28}

\accepted{2021-04-28}

\begin{document}

\publicationdetails{22}{2021}{2}{6}{6494}

\maketitle

\begin{abstract} Defant, Engen, and Miller defined a permutation to be \emph{uniquely sorted} if it has exactly one preimage under West's stack-sorting map. We enumerate classes of uniquely sorted permutations that avoid a pattern of length three and a pattern of length four by establishing bijections between these classes and various lattice paths. This allows us to prove nine conjectures of Defant.

\end{abstract}

\section{Introduction}

This paper is concerned with enumerating special classes of permutations that arise from West's \textit{stack-sorting map}. Stack-sorting was originally a computer algorithm on stacks introduced by \cite{knuth}. Later, \cite{west} defined the stack-sorting map, which we call $s$, as a deterministic variant of Knuth's algorithm. Since then, the map has been studied extensively: see \cite{c2, c3, c4, c5, bm, main, c8, c9, c10, preimages, c13, z}.

The stack-sorting map itself will not be relevant to the methods in this paper, but for sake of completeness, we include a simple recursive definition of it. The map $s$ sends the empty permutation to itself. For a permutation $\pi\in S_n$, we can write $\pi=LnR$ and define $s(\pi)=s(L)s(R)n$. For example, $s(516243)=s(51)s(243)6=s(1)5s(2)s(3)46=152346$.

The $\textit{fertility}$ of a permutation $\pi$ is the number of preimages of $\pi$ under $s$, or $|s^{-1}(\pi)|$. \cite{bm} defined a permutation to be \textit{sorted} if it has positive fertility, that is, if it has some preimage under $s$. 

Recently, \cite{stack} defined a permutation to be \textit{uniquely sorted} if its fertility is exactly 1, giving rise to a new and fruitful type of permutation that has a surprising amount of structure. We let $\mathcal{U}_n$ denote the set of uniquely sorted permutations in $S_n$. The work in \cite{stack} suggests that the relationship between uniquely sorted permutations and general permutations is analogous to the relationship between matchings and general set partitions. Moreover, the authors prove that the sets of uniquely sorted permutations of odd length (which we soon see are the only nonempty sets) are counted by an interesting sequence first introduced in \cite{lasalle} called \textit{Lassalle's sequence} (OEIS A180874).

A \textit{descent} of a permutation $\pi\in S_n$ is an index $i \in [n-1]$ where $\pi_i>\pi_{i+1}$, in which case we call $\pi_i$ a \textit{descent top} and $\pi_{i+1}$ a \textit{descent bottom}. Similarly, an \textit{ascent} of $\pi$ is an index $i \in [n-1]$ where $\pi_i<\pi_{i+1}$, in which case we call $\pi_i$ an \textit{ascent bottom} and $\pi_{i+1}$ an \textit{ascent top}. It follows that any permutation has an equal number of descents, descent tops, and descent bottoms as well as an equal number of ascents, ascent bottoms, and ascent tops. Moreover, for all $2 \leq i \leq n$, we have that $\pi_i$ is either an ascent top or a decent bottom (and, similarly, for all $1 \leq i \leq n-1$, $\pi_i$ is either an ascent bottom for a decent top). The following theorem characterizes uniquely sorted permutations and will serve as a basis for much of the work in this paper. 

\begin{Th}[\cite{stack}]\label{defantthm}
A permutation $\pi \in S_n$ is uniquely sorted if and only if it is sorted and it has exactly $\frac{n-1}{2}$ descents.
\end{Th}

As an immediate consequence, uniquely sorted permutations must have odd length, so from now on we will refer to the set $\mathcal{U}_{2k+1}$ instead of $\mathcal{U}_n$.

Let $\mathcal{U}_{2k+1}(\tau^{(1)}, \ldots, \tau^{(\ell)})$ be the set of permutations in $\mathcal{U}_{2k+1}$ that avoid all of the patterns $\tau^{(1)}, \ldots, \tau^{(\ell)}$ (see Section \ref{prelims} for a definition of pattern avoidance). In \cite{main}, combinatorial classes of this form with each $\tau^{(i)}$ length 3 are enumerated. This is done primarily through bijections between these classes and intervals of various posets on the set of Dyck paths of a fixed length.

In the same paper, Defant enumerates two classes of the form $\mathcal{U}_{2k+1}(\tau^{(1)}, \tau^{(2)})$, where $\tau^{(1)}$ has length 3 and $\tau^{(2)}$ has length 4, before making 18 conjectures concerning further interesting classes of this form. Each row of Table \ref{conjectures} represents the conjecture that the class of uniquely sorted permutations of odd length avoiding the given patterns is counted by the corresponding OEIS sequence.

{\tabulinesep=1.2mm
\begin{table}[H]
\centering

\begin{tabu}{|c|[1.5pt]c|}

    \hline Patterns & OEIS Sequence \\\tabucline[2.5pt]{-}
    $*312,1432$ &  \\
    $*312,2431$ &  \\
    $*312,3421$ & A001764 \\
    $*132,3412$ &  \\
    $*231,1423$ &  \\\hline
    $312,1243$ &  A122368 \\\hline
\end{tabu}
\begin{tabu}{|c|[1.5pt]c|}

    \hline Patterns & OEIS Sequence \\\tabucline[2.5pt]{-}
    $*132,3421$ &  \\
    $*132,4312$ & A001700 \\
    $231,1243$ &  \\\hline \noalign{\vskip -.08cm}  
    \makecell{\,\,$132,2341$ \vspace{.05cm}\\ $132,4123$} & \hspace{.13cm}A109081 \vspace{-.155cm} \\\hline
    $312,2341$ & A006605 \\\hline
\end{tabu}
\begin{tabu}{|c|[1.5pt]c|}

    \hline Patterns & OEIS Sequence \\\tabucline[2.5pt]{-}
    $312,3241$ & A279569 \\\hline
    $312,4321$ & A063020 \\\hline
    $132,4231$ & A071725 \\\hline
    $*231,1432$ & A001003 \\\hline
    $*231,4312$ & A127632 \\\hline
    $231,4321$ & A056010 \\\hline
\end{tabu}\vspace{.3cm}
\caption{Conjectural OEIS sequences enumerating sets of the form
$\mathcal{U}_{2k+1}(\tau^{(1)},\tau^{(2)})$.}\label{conjectures}
\end{table}}

In this paper, the conjectures concerning the nine notable pairs of patterns marked with an asterisk ($*$) in Table \ref{conjectures} are proven as theorems. The OEIS sequence A001700 consists of the binomial coefficients ${2k-1 \choose k}$, and the OEIS sequence A001764 consists of the 3-Catalan numbers $\frac{1}{2k+1}{3k \choose k}$, making these classes easier to enumerate since we have relatively simple closed formulas counting them. Additionally, the sequence A001003 counts the little Schr\"oder numbers, which have been well studied. The sequences A001764 and A127632 also count $\mathcal U_{2k+1}(312, 1342)$ and $\mathcal U_{2k+1}(231, 4132)$, respectively, as proven in \cite{main}. 

Starting in Section \ref{almost}, these nine classes are enumerated one by one. For each, the general structure of the permutations in the class is found by decomposition based on pattern avoidance. Then, using this structure, the class is enumerated, either by direct counting, bijection with a previously-counted class, generating functions, or, as the title of this paper suggests, bijection with certain types of lattice paths. Lattice paths were studied as early as the late $19^{\text{th}}$ century to solve \textit{Bertrand's ballot problem} (see \cite{ballot}) via a bijection with Dyck paths in \cite{dyck}. Since then, they have grown to become a crucial concept in math, computer science, statistics, and physics; see \cite{paths} for an overview of this history. Dyck paths, as mentioned before, are central to the methods in \cite{main} and will appear later in this paper. The bijections concerning lattice paths in this paper rely strongly on a fascinating natural analog between the conditions on uniquely sorted permutations and the conditions on certain types of lattice paths.

\section{Preliminaries}\label{prelims}
\subsection{Pattern Avoidance}

In this paper, a $\textit{permutation}$ of length $n$ is a sequence containing each element of the set $[n]=\{1,\ldots,n\}$ exactly once, written out in one-line notation. The \textit{symmetric group on n symbols}, denoted $S_n$, is the set of all permutations of length $n$. More generally, given a sequence of positive integers $a=a_1\cdots a_n$, any sequence of the form $a_{i_1}\cdots a_{i_\ell}$ with $i_1<\cdots<i_\ell$ is a \textit{subsequence} of $a$. If $i_1, \ldots, i_\ell$ are all consecutive, we call $a_{i_1}\cdots a_{i_\ell}$ a \textit{consecutive subsequence} of $a$.

The \textit{normalization} of a sequence of distinct positive integers $a=a_1\cdots a_n$ is the permutation in $S_n$ obtained by replacing the $i^\text{th}$ smallest entry of $a$ with $i$ for all $i$. For example, the normalization of 381 is 231. For a sequence of distinct positive integers $a$ and a permutation $\tau$, we say that $a$ \textit{contains} $\tau$ if there exists a subsequence of $a$ whose normalization is $\tau$. Otherwise, $a$ \textit{avoids} $\tau$. We let $\text{Av}(\tau^{(1)}, \ldots, \tau^{(\ell)})$ be the set of permutations avoiding all of the patterns $\tau^{(1)}, \ldots, \tau^{(\ell)}$.

\subsection{The Canonical Hook Configuration}

In order to make the sorted condition more workable, we introduce the concept of hooks. A permutation $\pi=\pi_1\cdots\pi_n$ can be visually represented via its \emph{plot}, which is the set of all points $(i,\pi_i)$ such that $i \in [n]$. These points can be connected via $\textit{hooks}$ of $\pi$. A hook $H$ is created by starting at a point $(i, \pi_i)$, which we call the \textit{southwest (SW) endpoint} of $H$, and then moving upward and then to the right to connect it to a second point $(j, \pi_j)$, which we call the \textit{northeast (NE) endpoint} of $H$. A point $(r, \pi_r)$ \textit{lies strictly below} $H$ if $i<r<j$ and $\pi_r<\pi_j$; it \textit{lies weakly below} $H$ if  $i<r\leq j$ and $\pi_r\leq\pi_j$.

Let $\pi$ have descents $d_1<\cdots<d_k$. The \textit{canonical hook configuration (CHC)}\footnote{Defant instead uses the term \textit{canonical valid hook configuration}; the present author drops the \textit{valid} to avoid wordiness.} of $\pi$ is the tuple $\mathcal{H} = (H_1, . . . , H_k)$ of hooks of $\pi$, defined as follows. First, the SW endpoint of the hook $H_i$
is $(d_i, \pi_{d_i})$. Let $\mathfrak{N}_i$ denote the NE endpoint of $H_i$. We determine these NE endpoints by starting with $\mathfrak{N}_k$, which is the leftmost point above and to the right of $(d_k, \pi_{d_k})$. Then, decrementing $i$ by one for each hook, $\mathfrak{N}_{i}$ is the leftmost point above and to the right of $(d_{i}, \pi_{d_{i}})$ that does not lie weakly below any of the hooks $H_{i+1}, \ldots, H_{k}$. If $\mathfrak{N}_{i}$ does not exist for some $i$, then $\pi$ does not have a CHC. An example of the CHC of a permutation is shown in Figure \ref{fig:CHC}.

\begin{figure}[H]
\centering
\begin{tikzpicture}[scale=.4]
\draw[fill] (1,2) circle [radius=0.2];
\draw[fill] (2,7) circle [radius=0.2];
\draw[fill] (3,3) circle [radius=0.2];
\draw[fill] (4,5) circle [radius=0.2];
\draw[fill] (5,9) circle [radius=0.2];
\draw[fill] (6,4) circle [radius=0.2];
\draw[fill] (7,8) circle [radius=0.2];
\draw[fill] (8,1) circle [radius=0.2];
\draw[fill] (9,6) circle [radius=0.2];
\draw[fill] (10,10) circle [radius=0.2];
\draw[fill] (11,11) circle [radius=0.2];
\draw[fill] (12,12) circle [radius=0.2];
\node [right] at (1,2) {2};
\node [right] at (2,7) {7};
\node [right] at (3,3) {3};
\node [right] at (4,5) {5};
\node [right] at (5,9) {9};
\node [right] at (6,4) {4};
\node [right] at (7,8) {8};
\node [right] at (8,1) {1};
\node [right] at (9,6) {6};
\node [right] at (10,10) {10};
\node [right] at (11,11) {11};
\node [right] at (12,12) {12};
\draw [thick] (2,7)--(2,9)--(5,9)--(5,11)--(11,11);
\draw[thick] (7,8)--(7,10)--(10,10);
\node [above] at (3.5, 9) {$H_1$};
\node [above] at (8, 11) {$H_2$};
\node [below] at (8.5, 10) {$H_3$};
\end{tikzpicture}
\caption{The CHC of 2\,7\,3\,5\,9\,4\,8\,1\,6\,10\,11\,12.}
\label{fig:CHC}
\end{figure}
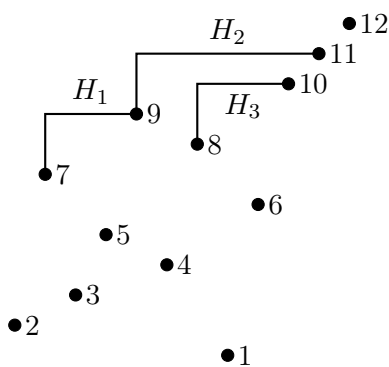

The following proposition allows us to determine whether a permutation is sorted using the CHC.

\begin{Prop}[\cite{preimages}]\label{CHC}
A permutation $\pi$ is sorted if and only if it has a canonical hook configuration.
\end{Prop}

This proposition, along with Theorem \ref{defantthm}, gives us that $\pi\in S_n$ is uniquely sorted if and only if it has a CHC and exactly $\frac{n-1}{2}$ descents. The permutation in Figure \ref{fig:CHC} has a CHC, so it is sorted. But it has length 12 and only 3 descents, so it is not uniquely sorted (in fact, its fertility is 160). We now introduce one particularly useful lemma.

\begin{Lemma}[\cite{main}]\label{partition}
Let $\pi \in \mathcal{U}_{2k+1}$, and let $\mathfrak{N}_1, \ldots, \mathfrak{N}_k$ be the NE endpoints of the hooks in $\pi$'s CHC. Let $\DB(\pi)$ be the set of descent bottoms of $\pi$. The two k-element sets $\DB(\pi)$ and $\{\mathfrak{N}_1, \ldots, \mathfrak{N}_k\}$ form a partition of $\{(i, \pi_i):2\leq i \leq 2k+1\}$.
\end{Lemma}

For example, the descent bottoms of the permutation in Figure \ref{fig:CHC} are the points $(3,3), (6,4),$ and  $(8,1)$, and the NE endpoints of hooks are $(5,9), (10,10),$ and $(11,11)$. Since the point $(7,8)$ is neither, the permutation is not uniquely sorted.  Note that this lemma tells us that in a uniquely sorted permutation, NE endpoints are precisely ascent tops and SW endpoints are precisely decent tops. Another immediate consequence is that the plot of any $\pi\in \mathcal{U}_{2k+1}$ must end with the point $(2k+1,2k+1)$.

\subsection{Permutation Structure}

Since we will be using Proposition \ref{CHC} to determine if a permutation is sorted, we will be dealing heavily with the plot of a permutation and thus will consider a permutation and its plot synonymously. Consequently, for a permutation $\pi$, we will refer to the entry $\pi_i$ and the point $(i, \pi_i)$ interchangeably. Note, also, that when we look at a plot, it suffices to consider only the relative order of the points, and not the actual positions of the points. 

For a plot containing two disjoint subsequences $\mu$ and $\lambda$, we say $\mu$ \textit{is above} $\lambda$ if every point in $\mu$ is above every point in $\lambda$. We define $\textit{below}$, $\textit{to the right of}$, and $\textit{to the left of}$ in a similar fashion.

Given the permutations $\mu$ and $\lambda$, their sum, denoted $\mu \oplus \lambda$, is the permutation obtained by placing the plot of $\lambda$ above and to the right of $\mu$. Their skew sum, $\mu \ominus \lambda$, is the permutation obtained by placing the plot of $\lambda$ below and to the right of $\mu$. Geometrically, we have:

\[\mu \oplus \lambda =\begin{array}{l}\begin{tikzpicture}[scale=.6]
\draw (-1,-1) rectangle (0,0);
\node at (-0.5,-0.5) {$\mu$};
\draw (0,0) rectangle (1,1);
\node at (0.5,0.5) {$\lambda$};
\end{tikzpicture}\end{array}\quad\text{and}\quad \mu \ominus \lambda =\begin{array}{l}\begin{tikzpicture}[scale=.6]
\draw (0,0) rectangle (1,1);
\node at (0.5, 0.5) {$\mu$};
\draw (1,0) rectangle (2,-1);
\node at (1.5,-0.5) {$\lambda$};
\end{tikzpicture}\end{array}.\]

Let $\Dec(n)=n(n-1)\cdots 21$ and $\Inc(n)=12\cdots (n-1)n$ denote the decreasing permutation of length $n$ and the increasing permutation of length $n$, respectively.  We denote a decreasing and increasing permutation, respectively, with the symbols:

\[\begin{array}{l}\begin{tikzpicture}[scale=1]
\draw[gray] (0,0) rectangle (1,1);
\draw[thick](0.8,0.2)--(0.2, 0.8);
\end{tikzpicture}\end{array}\quad\text{and}\quad\begin{array}{l}\begin{tikzpicture}[scale=1]
\draw[gray] (0,0) rectangle (1,1);
\draw[thick] (0.2,0.2)--(0.8,0.8);
\end{tikzpicture}\end{array}.\]

The \textit{tail length} of a
permutation $\pi= \pi_1\cdots\pi_n$, denoted $\tl(\pi)$, is the smallest nonnegative integer $\ell$ such that $\pi_{n-\ell} \neq n-\ell$. By convention, we let $\tl(\Inc(n))=n$.  If $\tl(\pi)=\ell$, then the \textit{tail} of $\pi$ is the list of points $(n - \ell + 1, n- \ell + 1), \ldots, (n, n)$.

\section{Three Bijections with Dyck Paths}

\subsection{Dyck Paths}A \textit{Dyck path} of semilength $k$ is a path starting at $(0,0)$ and ending at $(2k,0)$ that consists of $k$ $(1,1)$ steps (called \textit{up steps}) and $k$ $(1,-1)$ steps (called \textit{down steps}) and at no point crosses below the horizontal axis. We let $\mathcal{D}_k$ denote the set of Dyck paths of semilength $k$. It is a classical result that the sets $\mathcal{D}_k$ are counted by the $\textit{Catalan numbers}, C_k =\frac{1}{k+1}{2k \choose k}$. We can associate an \textit{up-down sequence} to a Dyck path by reading the path from left to right and recording the letter $U$ for an up step and $D$ for a down step. Note that the above-the-horizontal condition in a Dyck path is equivalent to every prefix of an up-down sequence having at least as many $U$'s as $D$'s. Going forward, we will treat a Dyck path and its up-down sequence as the same object.

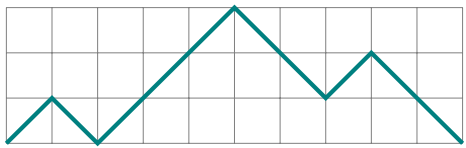
\begin{figure}[H]
\centering
\begin{tikzpicture}[scale=.6]
\draw[help lines] (0,0) grid (10,3);
\draw [ultra thick][teal] (0,0) -- (1,1) -- (2,0) -- (5,3)--(7,1)--(8,2)--(10,0);
\end{tikzpicture}
\caption{The Dyck path $\mathit{UDUUUDDUDD}$ of semilength 5.}
\end{figure}

\subsection{Three Classes of Permutations} We now highlight the structures of three particular classes of pattern-avoiding permutations:
\begin{itemize}
\item  A permutation $\pi$ avoids 132 and 231 if and only if $\pi=L1R$, where $L$ is decreasing and $R$ is increasing. We call this type of permutation $\textit{vee}$ after its resemblance to the letter V. We call the imaginary vertical line at 1, separating $L$ and $R$, the \textit{vertical} of the permutation.

\item A permutation $\pi$ avoids 132 and 312 if and only if the plot of $\pi$ is the 90 degree clockwise rotation of the plot of a \textit{vee} permutation. We call this type of permutation $svee$, which is short for sideways-vee. We call the rotation of the vertical the \textit{horizontal}.

\item A permutation $\pi$ avoids 231 and 312 if and only if $\pi$ is the sum of decreasing permutations; such a permutation is called \emph{layered}. Each decreasing permutation in the sum is called a $\textit{layer}$.

\end{itemize}

As we will see in later sections, these three types of permutations will be key for decomposing permutations of the form $\mathcal{U}_{2k+1}(\tau^{(1)}, \tau^{(2)})$. We will denote a vee, svee, and layered permutation, respectively, with the symbols: 

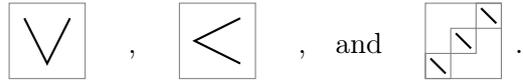
\[\begin{array}{l}\begin{tikzpicture}[scale=1]
\draw[gray] (0,0) rectangle (1,1);
\draw[thick](0.2,0.8)--(0.5,0.2)--(0.8,0.8);
\end{tikzpicture}\end{array}\quad\text{,}\quad\begin{array}{l}\begin{tikzpicture}[scale=1]
\draw[gray] (0,0) rectangle (1,1);
\draw[thick] (0.8,0.8)--(0.2,0.5)--(0.8,0.2);
\end{tikzpicture}\end{array}\quad\text{,}\quad    \text{and}\quad\begin{array}{l}\begin{tikzpicture}[scale=.333]
\draw[gray] (9.6,9.6) rectangle (12.6, 12.6);
\draw[gray] (9.6, 9.6) rectangle (10.6, 10.6);
\draw[gray] (10.6, 10.6) rectangle (11.6, 11.6);
\draw[gray] (11.6, 11.6) rectangle (12.6, 12.6);
\draw[thick]  (9.8, 10.4)--(10.4, 9.8);
\draw[thick]  (10.8, 11.4)--(11.4, 10.8);
\draw[thick]  (11.8, 12.4)--(12.4, 11.8);
\end{tikzpicture}\end{array}.\]

\subsection{Rethinking Some Past Results}\label{past}

In \cite{main}, Defant proves that
\[|\mathcal{U}_{2k+1}(132, 231)|=|\mathcal{U}_{2k+1}(132, 312)|=|\mathcal{U}_{2k+1}(231, 312)|=C_k.\]
His enumeration of the first class derives from the following lemma.

\begin{Lemma}[\cite{main}]\label{lemma:defantdyck}
There exists a bijection $\mathcal{U}_{2k+1}(132, 231)\to \mathcal{D}_k$.
\end{Lemma}

\begin{proof}
Given $\pi \in \mathcal{U}_{2k+1}(132, 231)$, we have from above that $\pi$ can be written as $\pi=L1R$, where $L$ is decreasing and $R$ is increasing. We construct the path $\Lambda=\Lambda_1\cdots\Lambda_{2k}$ as follows: let $\Lambda_i=U$ if $2k+2-i$ is an entry in $R$ and $\Lambda_i=D$ if $2k+2-i$ is an entry in $L$. The permutation $\pi$ is uniquely sorted and thus has exactly $k$ descents, causing $|L|=|R|=k$. This means that $\Lambda$ has $k$ ups and $k$ downs. Moreover, since $\pi$ has a CHC, every prefix of $\Lambda$ contains at least as many $U$'s as $D$'s. Thus, $\Lambda\in \mathcal{D}_k$. This map is easily reversible, so this is indeed a bijection.
\end{proof}

\begin{figure}[H]\label{dyck}
\[\begin{array}{l}
\begin{tikzpicture}[scale=.35]
\draw[fill](1,10) circle [radius=0.2];
\draw[fill] (2,6) circle [radius=0.2];
\draw[fill] (3,5) circle [radius=0.2];
\draw[fill] (4,3) circle [radius=0.2];
\draw[fill] (5,2) circle [radius=0.2];
\draw[fill] (6,1) circle [radius=0.2];
\draw[fill] (7,4) circle [radius=0.2];
\draw[fill] (8,7) circle [radius=0.2];
\draw[fill] (9,8) circle [radius=0.2];
\draw[fill] (10,9) circle [radius=0.2];
\draw[fill] (11,11) circle [radius=0.2];
\draw [thick] (1,10)--(1,11)--(11,11);
\draw [thick] (2,6)--(2,9)--(10,9);
\draw [thick] (3,5)--(3,8)--(9,8);
\draw [thick] (4,3)--(4,7)--(8,7);
\draw[thick] (5,2)--(5,4)--(7,4);
\end{tikzpicture}\end{array}
\begin{array}{l}
\begin{tikzpicture}[scale=.8]
\draw [thick] (0,0)--(1,0);
\draw [thick] (0.2, -0.2)--(0,0)--(0.2,0.2);
\draw [thick] (0.8, -0.2)--(1,0)--(0.8, 0.2);
\draw [white] (1.1,0)--(1.2,0);
\end{tikzpicture}
\end{array}\begin{array}{l} \begin{tikzpicture}[scale=.5]
\draw[help lines] (0,0) grid (10,3);
\draw [ultra thick][teal] (0,0) -- (1,1) -- (2,0) -- (5,3)--(7,1)--(8,2)--(10,0);
\end{tikzpicture}\end{array}\] 
\caption{The vee permutation 10\,6\,5\,3\,2\,1\,4\,7\,8\,9\,11 and its image under the bijection in Lemma \ref{lemma:defantdyck}, the Dyck path $\mathit{UDUUUDDUDD}$.}
\end{figure}
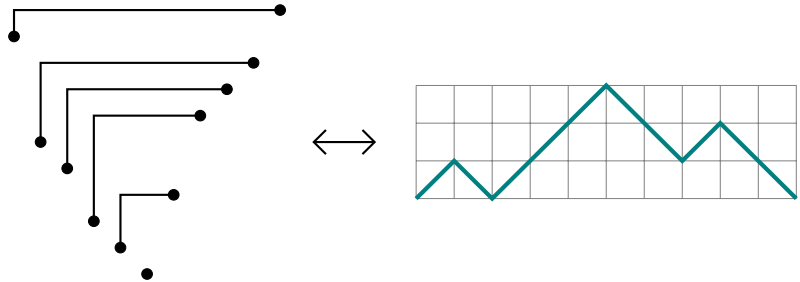

This simple bijection touches upon an intriguing connection between Dyck paths and uniquely sorted permutations. While Defant uses a different map for proving that: \[|\mathcal{U}_{2k+1}(132, 312)|=|\mathcal{U}_{2k+1}(231, 312)|=C_k,\] we note here that a method similar to that in the proof of Lemma \ref{lemma:defantdyck} can be used to make bijections with these sets.

\begin{Lemma}\label{svee}
There exists a bijection $\mathcal{U}_{2k+1}(132, 312)\to \mathcal{D}_k$.
\end{Lemma}

\begin{proof}
Given $\pi \in \mathcal{U}_{2k+1}(132, 312)$, we have from above that $\pi$ is svee. We construct the path $\Lambda=\Lambda_1\cdots\Lambda_{2k}$ as follows: for $2\leq i \leq 2k+1$, if $\pi_i>\pi_{i-1}$, let $\Lambda_{i-1}=D$, and if $\pi_i<\pi_{i-1}$, let $\Lambda_{i-1}=U$. Note that we are simply associating ascent tops with $D$ and descent bottoms with $U$. The permutation $\pi$ is uniquely sorted and thus has exactly $k$ descents, causing $|L|=|R|=k$. This means that $\Lambda$ has $k$ ups and $k$ downs. Moreover, since $\pi$ has a CHC, the NE endpoint of a hook (which must be an ascent top) comes after the SW endpoint of that same hook (which must be a descent top), and thus every prefix of $\Lambda$ contains at least as many $U$'s as $D$'s. Thus, $\Lambda\in \mathcal{D}_k$. This map is easily reversible since ascent tops must be above the horizontal and decent bottoms must be below, so this is indeed a bijection.
\end{proof}

\begin{Lemma}\label{layered}
There exists a bijection $\mathcal{U}_{2k+1}(231, 312)\to \mathcal{D}_k$.
\end{Lemma}

\begin{proof}
Given $\pi \in \mathcal{U}_{2k+1}(231, 312)$, we have from above that $\pi$ is layered. We construct the path $\Lambda=\Lambda_1\cdots\Lambda_{2k}$ as follows: for $2\leq i\leq 2k+1$, let $\Lambda_{i-1}=D$ if $\pi_i>\pi_{i-1}$, and let $\Lambda_{i-1}=U$ if $\pi_i<\pi_{i-1}$. Note that we are then simply associating ascent tops with $D$ and descent bottoms with $U$. The rest follows as in the proof of the previous lemma (only know ascent tops are the first entries in each layer besides the first).
\end{proof}

The above lemmas not only help us to rethink some of the results of \cite{main}, but demonstrate the simplest application of a method crucial to enumerating classes in this paper. More specifically, the condition that a uniquely sorted permutation of length $2k+1$ has $k$ descents is akin to Dyck paths of semilength $k$ having $k$ $U$'s, and the existence of a CHC is akin to every prefix of a Dyck path having at least as many $U$'s as $D$'s. In Sections \ref{motzkin} and \ref{schroder}, we will see how to translate this idea to more complex permutation classes and paths.

\section{Two Almost-Vee Classes}\label{almost}

In this section, we prove that, out of the eighteen conjectured classes, two are counted nicely by the binomial coefficient ${2k-1\choose k}$. 

\begin{Th}\label{modveethm} 
We have $|\mathcal{U}_{2k+1}(132,4312)|={2k-1\choose k}.$
\end{Th}
\begin{proof}

Consider the plot of some $\pi \in \mathcal{U}_{2k+1}(132,4312)$, and let $\rho$ be any consecutive subsequence of $\pi$, which, by definition, also avoids 132 and 4312.

Since $\rho$ avoids 132, we can decompose $\rho = \lambda\rho_m\mu$, where $\rho_m$ is the largest entry in $\rho$ and $\lambda$ is above $\mu$. Now we consider two cases.\\

\noindent\underline{Case 1:} The subsequence $\mu$ contains the pattern 12. Then $\lambda$ must not contain the pattern 21, since $\rho$ avoids 4312. This implies $\lambda$ is increasing. Moreover, since $\rho$ avoids 4312, $\mu$ must avoid 312, so $\mu$ avoids both 132 and 312 and thus is svee. Thus, we can write $\rho = I \ominus \tau$ where $I$ is increasing and $\tau$ is svee.\\

\noindent\underline{Case 2:} The subsequence $\mu$ is empty or decreasing. In this case, there are no clear restrictions on $\lambda$.\\

We first perform the above decomposition on $\pi$. In case 1, we are done. In case 2, we can repeat the same decomposition process on $\lambda$ instead of $\pi$. This can be repeated until we are in case 1 or until the pieces in the decomposition are empty. Note that in case 2 we add a point to the right and above the unknown portion of $\pi$ as well as a decreasing permutation to the right and below the unknown portion of $\pi$, which preserves a svee shape in $\pi$. In case 1 we add the skew sum of an increasing permutation $I$ and a svee $\tau$, which ends the decomposition process. However, if the size of $I$ is at least 2, the second element of $I$ is an ascent top but not a NE endpoint (since there are no descents to the left of it), contradicting Lemma \ref{partition}. Thus $I$ is a single point, so $\pi$ can be thought of as a svee permutation $\sigma$ preceded by a special point that lies above the horizontal of the svee. Thus $\pi=\pi_1\sigma$, where $\pi_1$ lies above $\sigma$'s horizontal (if $\pi_1$ lies \textit{immediately} above the horizontal, $\pi$ is simply svee). We call such a permutation a modified svee, or \textit{modsvee}, for short. See Figure \ref{modsvee} for an example of a modsvee permutation.

\begin{figure}[H]\label{modsvee}
\centering
\begin{tikzpicture}[scale=.5]
\draw[fill][red] (1,6) circle [radius=0.2];
\draw[fill][blue] (2,4) circle [radius=0.2];
\draw[fill][blue] (3,3) circle [radius=0.2];
\draw[fill][blue] (4,5) circle [radius=0.2];
\draw[fill][blue](5,2) circle [radius=0.2];
\draw[fill][blue](6,7) circle [radius=0.2];
\draw[fill][blue](7,8) circle [radius=0.2];
\draw[fill][blue](8,1) circle [radius=0.2];
\draw[fill][blue](9,9) circle [radius=0.2];
\draw [thick](1,6)--(1,8)--(7,8)--(7,9)--(9,9);
\draw[thick] (2,4)--(2,5)--(4,5)--(4,7)--(6,7);
\draw[thick][dashed][gray] (0.5,0.5) rectangle (8.5, 8.5);
\draw[thick][dashed][gray] (7.5,0.5) rectangle (8.5, 1.5);
\draw[thick][dashed][gray] (0.5,1.5) rectangle (6.5, 7.5);
\draw[thick][dashed][gray] (0.5,1.5) rectangle (5.5, 6.5);
\draw[thick][dashed][gray] (1.5,1.5) rectangle (5.5, 6.5);

\end{tikzpicture}
\caption{The modsvee permutation 643527819. The dashed boxes illustrate which pieces of the permutation are decomposed after each step in the decomposition process.}
\end{figure}

Following \cite{main}, we call a uniquely sorted permutation $\pi \in \mathcal{U}_{2k+1}$ \textit{nice} if the SW endpoint of the hook in the CHC of $\pi$ with NE endpoint $(2k+1, 2k+1)$ is $(1, \pi_1)$. For a nice $\pi\in \mathcal{U}_{2k+1}(132,4312)$, the hook $H_1$ connecting $(1, \pi_1)$ to $(2k+1, 2k+1)$ lies above every other hook of $\pi$ and thus does not interfere with the rest of the CHC. Thus, $\pi$ can be decomposed into the hook $H_1$, plus the svee permutation of size $2k-1$ whose CHC consists of the remaining hooks of $\pi$. Recall from Lemma \ref{svee} that there are $C_{k-1}$ such svee permutations. The NE endpoint of $H_1$ must be $(2k+1, 2k+1)$, whereas the height of the SW endpoint $(1, \pi_1)$ can be immediately below any of the $k$ ascent tops of the svee. Thus, there are $kC_{k-1}$ nice permutations in $\mathcal U_{2k+1}(132,4312)$.

Now, given any (not necessarily nice) permutation $\pi \in\mathcal{U}_{2k+1}(132,4312)$, we have that $(1,\pi_1)$ lies above the descending part of the svee shape following it, so 1 is a descent and thus $(1,\pi_1)$ is attached via a hook to one of the $k$ points in the ascending part of the svee. Fix some $1\leq j \leq k$ and let $(1, \pi_1)$ be attached to the $j^\text{th}$ point in the ascending part of svee, which we will call $(m, \pi_m)$. Since any hook with SW endpoint to the left of $(m, \pi_m)$ cannot intersect the hook with NE endpoint $(m, \pi_m)$, all of the hooks before this point are entirely contained to the left of it. Thus, the points $(1,\pi_1)$ through $(m, \pi_m)$ form a nice uniquely sorted modsvee permutation of size $2j+1$, of which there are $jC_{j-1}$, from above. It follows that $(m, \pi_m)= (2j+1,2j+1)$. The remaining hooks other than the one with SW endpoint $(2j+1,2j+1)$ lie to the right of the point $(2j+1,2j+1)$.  Thus, the points $(2j+1, 2j+1)$ through $(2k+1, 2k+1)$ form a uniquely sorted permutation of size $2(k-j)+1$ that is also svee shaped. From Lemma \ref{svee}, there are $C_{k-j}$ such permutations. Thus, for a given $j$, there are $jC_{j-1}C_{k-j}$ permutations. Summing over all possible $j$ gives us that $|\mathcal{U}_{2k+1}(132,4312)|=\sum_{j=1}^k jC_{j-1}C_{k-j}$. It is routine to show that $\sum_{j=1}^k jC_{j-1}C_{k-j}={2k-1\choose k}$, giving us the desired result. See Figure \ref{m} for an example of this decomposition. \end{proof}

\begin{figure}[H]
\centering
\begin{tikzpicture}[scale=.5]
\draw[fill][red] (1,6) circle [radius=0.2];
\draw[fill][blue] (2,4) circle [radius=0.2];
\draw[fill][blue] (3,3) circle [radius=0.2];
\draw[fill][blue] (4,5) circle [radius=0.2];
\draw[fill][blue](5,2) circle [radius=0.2];
\draw[fill][blue](6,7) circle [radius=0.2];
\draw[fill][blue](7,8) circle [radius=0.2];
\draw[fill][blue](8,1) circle [radius=0.2];
\draw[fill][blue](9,9) circle [radius=0.2];
\draw [thick](1,6)--(1,8)--(7,8)--(7,9)--(9,9);
\draw[thick] (2,4)--(2,5)--(4,5)--(4,7)--(6,7);
\draw[ultra thick][dashed][gray] (0.5,1.5) rectangle (7.5, 8.5);
\draw[ultra thick][dashed][gray] (6.5,0.5) rectangle (9.5, 9.5);
\end{tikzpicture}
\caption{The decomposition of the modsvee permutation 643527819 into a nice modsvee uniquely sorted permutation (on the left) and a svee uniquely sorted permutation (on the right).}
\label{m}
\end{figure}
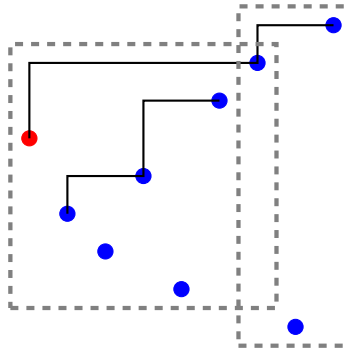

\begin{Th} We have
$|\mathcal{U}_{2k+1}(132,3421)|={2k-1\choose k}.$
\end{Th}

\begin{proof}

The permutation obtained by reflecting the plot of $\pi$ through the line $y=x$ is called the \textit{inverse}\footnote{When $\pi$ is considered as an element of the symmetric group, we have that $\pi^{-1}$ is indeed the algebraic inverse of $\pi$.} of $\pi$, denoted $\pi^{-1}$. Note that $132$ is its own inverse and that the inverse of $3421$ is $4312$. Therefore, a permutation avoids $132$ and $3421$ if and only if its inverse avoids $132$ and $4312$. Also, it is proven in Lemma 4.3 in \cite{main} that a permutation that avoids 132 has the same number of descents as its inverse. It is natural, then, to conjecture that permutations in $\mathcal{U}_{2k+1}(132,3421)$ are simply the inverses of those in $\mathcal{U}_{2k+1}(132,4312)$. Note that this does not hold for general permutations; for example, the permutation 31425 is uniquely sorted, but its inverse is 24135, which is not uniquely sorted. 

Consider some $\pi \in \mathcal{U}_{2k+1}(132,3421)$ and some consecutive subsequence of $\pi$, called $\rho$. Again, since $\rho$ avoids 132, we can decompose $\rho=\lambda \rho_m \mu$, where $\rho_m$ is the largest entry in $\rho$ and $\lambda$ is above $\mu$. There are three cases.\\

\underline{Case 1:} The subsequences $\lambda$ and $\mu$ are both nonempty. Then $\lambda\rho_{m}$ contains the pattern 12, so $\mu$ cannot contain the pattern 21, so $\mu$ is increasing. Then since nonempty $\mu$ lies below and to the right of $\lambda$, we have that $\lambda$ must avoid 231 in order for $\rho$ to avoid 3421. Thus $\lambda$ avoids both 132 and 231 and therefore is vee, giving us the decomposition $\rho= \tau\ominus I$, where $\tau$ is vee and $I$ is increasing.\\

\underline{Case 2:} The subsequence $\lambda$ is empty, and there are no clear restrictions on $\mu$.\\

\underline{Case 3:} The subsequence $\mu$ is empty, and there are no clear restrictions on $\lambda$.\\

We first perform the above decomposition on $\pi$. If we are in case 1, the structure of $\pi$ is determined to a degree to which we are satisfied, and we are done. In the other cases, we can repeat the same decomposition process on the unknown portion of the permutation until we end up in case 1 or all of the unknown portions are empty. Note that cases 2 and 3 preserve a vee shape in $\pi$, whereas in case 1 we add the skew sum of a vee permutation and an increasing permutation $I$, which ends the decomposition process. However, if the size of $I$ is at least 2, the second element of $I$ is an ascent top but not a NE endpoint (since there are no descents below it), contradicting Lemma \ref{partition}. Thus $I$ is a single point, so the result of this process is the decomposition consisting of a vee permutation, plus one special point below the vee that is to the right of the vertical of the vee (if the point is $\textit{immediately}$ to the right the vertical, $\pi$ is simply vee). We call such a permutation a modified vee, or $\textit{modvee}$ for short. 

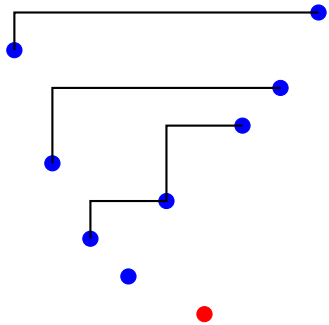
\begin{figure}[H]\label{modvee}
\centering
\begin{tikzpicture}[scale=.5]
\draw[fill][blue] (1,8) circle [radius=0.2];
\draw[fill][blue] (2,5) circle [radius=0.2];
\draw[fill][blue] (3,3) circle [radius=0.2];
\draw[fill][blue] (4,2) circle [radius=0.2];
\draw[fill][blue](5,4) circle [radius=0.2];
\draw[fill][red](6,1) circle [radius=0.2];
\draw[fill][blue](7,6) circle [radius=0.2];
\draw[fill][blue](8,7) circle [radius=0.2];
\draw[fill][blue](9,9) circle [radius=0.2];
\draw [thick](3,3)--(3,4)--(5,4)--(5,6)--(7,6);
\draw [thick] (2,5)--(2,7)--(8,7);
\draw [thick] (1,8)--(1,9)--(9,9);
\end{tikzpicture}
\caption{The modvee permutation 853241679, which is the inverse of the permutation in Figure \ref{m}.}
\end{figure}

As hoped for, it follows from the definition of modvee and modsvee that the inverse of a permutation in one class is in the other class. Moreover, note that when we reflect a modvee permutation, all of the points to the left of the vertical that were descent tops become descent bottoms, the points to the right of the vertical that were ascent tops remain ascent tops, and the special point becomes the special point in the resulting modsvee permutation. Thus, the relative order of ascent tops and descent bottoms is the same in the modvee permutation, read bottom to top, as it is in its image modsvee permutation, read left to right. Then, by the logic in the lemmas in Section \ref{past}, a modvee permutation has a CHC if and only if its inverse does. Thus, inversion does indeed define a bijection between the two classes, giving us that $|\mathcal{U}_{2k+1}(132,3421)|=|\mathcal{U}_{2k+1}(132,4312)|={2k-1\choose k}$. \end{proof}

\section{Bijections with S-Motzkin Paths}\label{motzkin}

A variant of a Dyck path is a \textit{Motzkin path} of length $k$, which is a path from $(0,0)$ to $(k,0)$ that consists of up steps, down steps, and (1,0) steps (called \textit{east steps}) and at no point crosses below the horizontal axis. In the up-down sequence of a Motzkin path, we write the letter $E$ for east steps. \cite{helmut} defined a specific subclass of Motzkin paths.

\begin{Def}
An $\textbf{S}$-Motzkin path is a Motzkin path with $k$ up steps, $k$ down steps, and $k$ east steps such that:
\begin{enumerate}
    \item The first step is east.
    \item Between every two east steps is exactly one up step.
    \item The $i^\text{th}$ down step must occur after at least $i$ east steps and $i$ up steps. This is equivalent to the condition that the path does not cross the horizontal axis.
\end{enumerate}
\end{Def}

Interestingly, this definition arose from a 2018 International Mathematics Competition question proposed in \cite{imo} about a frog moving through three-space.

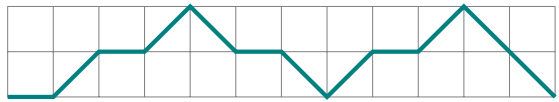
\begin{figure}[H]
\centering
\begin{tikzpicture}[scale=.6]
\draw[help lines] (0,0) grid (12,2);
\draw [ultra thick][teal] (0,0) -- (1,0)--(2,1)--(3,1)--(4,2)--(5,1)--(6,1)--(7,0)--(8,1)--(9,1)--(10,2)--(12,0);
\end{tikzpicture}
\caption{The $\bf{S}$-Motzkin path $\mathit{EUEUDEDUEUDD}$ of length 12.}
\end{figure}

Let $\mathcal M_k^{\bf S}$ denote the set of $\bf{S}$-Motzkin paths of length $3k$.

\begin{Th}[\cite{helmut}]
We have $|\mathcal M_k^{\bf S}|=\frac{1}{2k+1}{3k\choose k}$.
\end{Th}

Prodinger and Selkirk prove this with a bijection between $\bf{S}$-Motzkin paths of length $3k$ and ternary trees on $k$ nodes, the latter of which is counted by the so-called 3-Catalan numbers $\frac{1}{2k+1}{3k\choose k}$; see \cite{tree}. These paths will be crucial for enumerating five classes of uniquely sorted permutations counted by the same formula.

\begin{Th}\label{stair-svee}
We have $|\mathcal{U}_{2k+1}(312, 2431)|=\frac{1}{2k+1}{3k\choose k}$.
\end{Th}
\begin{proof}

Consider the plot of some $\pi \in \mathcal{U}_{2k+1}(312, 2431)$ and some consecutive subsequence of $\pi$, called $\rho$. Since $\rho$ avoids 312, we can write $\rho=\lambda \rho_{m} \mu$, where $\rho_{m}$ is the smallest entry in $\rho$ and  $\lambda$ is below $\mu$. Since $\rho$ avoids 2431 and $\rho_{m}$ lies below and to the right of $\lambda$, we have that $\lambda$ avoids both 312 and 132 and thus is svee.

We first perform this decomposition on $\pi$. Once we find svee $\lambda$, we can repeat this same decomposition process on $\mu$ instead of $\pi$ and continue to repeat until the subsequences are empty. The result is the decomposition $\pi =( \lambda^{(1)}\ominus1)\oplus (\lambda^{(2)}\ominus 1)\oplus \cdots \oplus(\lambda^{(l-1)}\ominus1)\oplus(\lambda^{(\ell)})$, where each $\lambda^{(i)}$ is svee. We call a permutation of this form $\textit{stair-svee}$, and we call each $\lambda^{(i)}\ominus1$ (or, for the last $i$, $\lambda^{(\ell)}$) a $\textit{block}$ of $\pi$. By default, we assume blocks are maximal: that is, starting at the leftmost block in the permutation, each block is as large as it can be while still the correct shape.

\begin{figure}[H]
\centering
\begin{tikzpicture}[scale=.3]
\draw[gray] (0,0) rectangle (4,4);
\draw [thick] (2.7, 1.3)--(0.3, 2.5)--(2.7, 3.7);
\draw[fill] (3.5, 0.5) circle [radius=0.1];
\draw [gray](4,4) rectangle (8,8);
\draw [thick] (6.7, 5.3)--(4.3, 6.5)--(6.7, 7.7);
\draw[fill] (7.5, 4.5) circle [radius=0.1];
\draw[fill] (8.4,8.4) circle [radius=0.025];
\draw[fill] (8.8, 8.8) circle [radius=0.025];
\draw[fill] (9.2,9.2) circle [radius=0.025];
\draw[gray] (9.6,9.6) rectangle (12.6, 12.6);
\draw [thick] (12.3, 9.9)--(9.9, 11.1)--(12.3, 12.3);

\end{tikzpicture}
\caption{The structure of a stair-svee permutation.}
\end{figure}
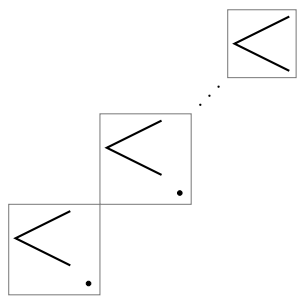

Now, we define a bijection from $\mathcal{U}_{2k+1}(312, 2431)$ to $\mathcal M_k^{\bf S}$ via the following rule. Let $\Lambda$ be the path $EUEU\ldots EU$ containing alternating $k$ $E$'s and $k$ $U$'s. Given $\pi\in\mathcal{U}_{2k+1}(312,2431)$, let $a_1, \ldots, a_k$ be the ascent tops of $\pi$, ordered from left to right, and let $n_i$ be the number of descent bottoms to the left of $a_i$. For $1\leq i \leq k$, if $a_i$ is in the same block as the point immediately to the left of it, insert a $D$ immediately after the $n_i^\text{th}$ $U$ in $\Lambda$.  If $a_i$ is not in the same block as the point immediately to the left of it, insert a $D$ immediately after the $n_{i}+1^\text{th}$ $E$ in $\Lambda$.

By construction, $\Lambda$ begins with $E$ and alternates between $E$ and $U$. Note that by the same logic as in the proofs of the lemmas in Section \ref{past}, the $k^\text{th}$ $D$ in $\Lambda$ appears after at least $k$ $U$'s and $k$ $E$'s. The particular placement of $D$'s allows us to recover the type of ascent that created it. The only possible issue is that with a path containing $\mathit{UDED}$ it is ambiguous which type of ascent came first. However, since $\pi$ is stair-svee, in which each block must end in a descent, an ascent top that is immediately preceded by a point in the same block cannot be followed by an ascent top in a new block. Therefore this case is actually unambiguous. Thus, this map is indeed a bijection and $|\mathcal{U}_{2k+1}(312, 2431)|=| \mathcal M_k^{\bf S}|=\frac{1}{2k+1}{3k\choose k}$.\end{proof}

\begin{figure}[H]\label{stair-svee bij}
\[\begin{array}{l}
\begin{tikzpicture}[scale=.35]
\draw[fill](1,3) circle [radius=0.2];
\draw[fill] (2,2) circle [radius=0.2];
\draw[fill] (3,4) circle [radius=0.2];
\draw[fill] (4,1) circle [radius=0.2];
\draw[fill] (5,9) circle [radius=0.2];
\draw[fill] (6,8) circle [radius=0.2];
\draw[fill] (7,7) circle [radius=0.2];
\draw[fill] (8,10) circle [radius=0.2];
\draw[fill] (9,11) circle [radius=0.2];
\draw[fill] (10,6) circle [radius=0.2];
\draw[fill] (11,5) circle [radius=0.2];
\draw[fill] (12,12) circle [radius=0.2];
\draw[fill] (13,13) circle [radius=0.2];
\draw [thick] (10,6)--(10,12)--(12,12);
\draw [thick] (1,3)--(1,4)--(3,4)--(3,9)--(5,9)--(5,11)--(9,11)--(9,13)--(13,13);
\draw [thick] (6,8)--(6,10)--(8,10);
\draw[dashed][thick][gray] (0.5,0.5) rectangle (4.5, 4.5);
\draw[dashed][thick][gray] (4.5,4.5) rectangle (13.5, 13.5);
\end{tikzpicture}\end{array}
\begin{array}{l}
\begin{tikzpicture}[scale=.7]
\draw [thick] (0,0)--(1,0);
\draw [thick] (0.2, -0.2)--(0,0)--(0.2,0.2);
\draw [thick] (0.8, -0.2)--(1,0)--(0.8, 0.2);
\draw [white] (1.1,0)--(1.2,0);
\end{tikzpicture}
\end{array}\begin{array}{l} \begin{tikzpicture}[scale=.45]
\draw[help lines] (0,0) grid (18,2);
\draw [ultra thick][teal] (0,0) -- (1,0)--(2,1)--(3,0)--(4,0)--(5,1)--(6,1)--(7,0)--(8,1)--(9,1)--(10,2)--(12,0)--(13,0)--(14,1)--(15,1)--(16,2)--(18,0);
\end{tikzpicture}\end{array}\] 
\caption{The stair-svee permutation 3\,2\,4\,1\,9\,8\,7\,10\,11\,6\,5\,12\,13 and its image under the bijection in Theorem \ref{stair-svee}, the $\bf{S}$-Motzkin path $\mathit{EUDEUEDUEUDDEUEUDD}$.}
\end{figure}
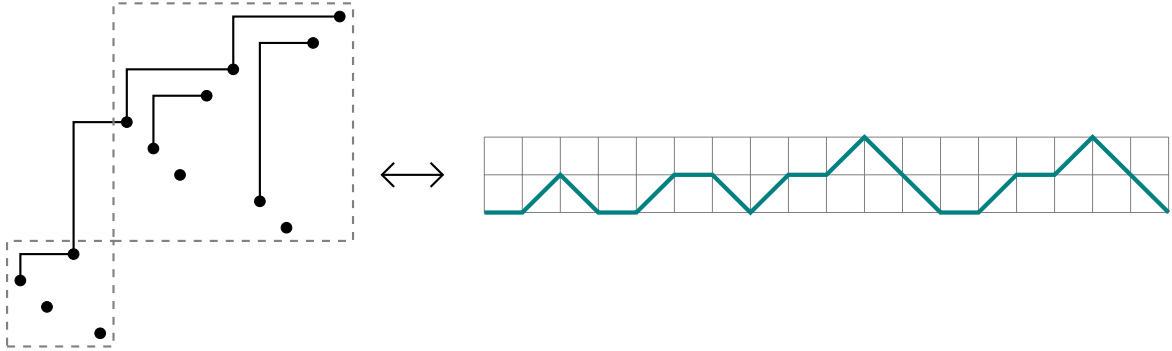

\begin{Th}\label{stair-layered}
We have $|\mathcal{U}_{2k+1}(312, 3421)|=\frac{1}{2k+1}{3k\choose k}$.
\end{Th}
\begin{proof}

Consider the plot of some $\pi \in \mathcal{U}_{2k+1}(312, 3421)$ and some consecutive subsequence of $\pi$, called $\rho$. Since $\rho$ avoids 312, we can write $\rho=\lambda \rho_m \mu$, where $\rho_m$ is the smallest entry in $\rho$ and $\lambda$ is below $\mu$. Since $\rho$ avoids 3421 and $\rho_{m}$ lies below and to the right of $\lambda$, we have that $\lambda$ avoids both 312 and 231 and thus is layered. 

We first perform the above decomposition on $\pi$. Once we find layered $\lambda$, we can repeat this same decomposition process on $\mu$ instead of $\pi$ and continue to repeat until the subsequences are empty. The result is the decomposition $\pi =( \lambda^{(1)}\ominus1)\oplus (\lambda^{(2)}\ominus 1)\oplus \cdots \oplus(\lambda^{(l-1)}\ominus1)\oplus(\lambda^{(\ell)})$, where each $\lambda^{(i)}$ is layered. We call a permutation of this form $\textit{stair-layered}$, and again we call each $\lambda^{(i)}\ominus1$ (or, for the last $i$, $\lambda^{(\ell)}$) a $\textit{block}$ of $\pi$.

\begin{figure}[H]
\centering
\begin{tikzpicture}[scale=.25]
\draw[gray] (0,0) rectangle (4,4);
\draw[gray] (0,1) rectangle (3,4);
\draw[gray] (0,1) rectangle (1,2);
\draw[gray] (1,2) rectangle (2,3);
\draw[gray] (2,3) rectangle (3,4);
\draw[thick] (0.2, 1.8)--(0.8, 1.2);
\draw[thick] (1.2, 2.8)--(1.8, 2.2);
\draw[thick] (2.2, 3.8)--(2.8, 3.2);
\draw[fill] (3.5, 0.5) circle [radius=0.1];
\draw [gray](4,4) rectangle (8,8);

\draw[gray] (4,5) rectangle (7,8);
\draw[gray] (4,5) rectangle (5,6);
\draw[gray] (5,6) rectangle (6,7);
\draw[gray] (6,7) rectangle (7,8);
\draw[thick] (4.2, 5.8)--(4.8, 5.2);
\draw[thick] (5.2, 6.8)--(5.8, 6.2);
\draw[thick] (6.2, 7.8)--(6.8, 7.2);

\draw[fill] (7.5, 4.5) circle [radius=0.1];
\draw[fill] (8.4,8.4) circle [radius=0.025];
\draw[fill] (8.8, 8.8) circle [radius=0.025];
\draw[fill] (9.2,9.2) circle [radius=0.025];
\draw[gray] (9.6,9.6) rectangle (12.6, 12.6);
\draw[gray] (9.6, 9.6) rectangle (10.6, 10.6);
\draw[gray] (10.6, 10.6) rectangle (11.6, 11.6);
\draw[gray] (11.6, 11.6) rectangle (12.6, 12.6);
\draw[thick]  (9.8, 10.4)--(10.4, 9.8);
\draw[thick]  (10.8, 11.4)--(11.4, 10.8);
\draw[thick]  (11.8, 12.4)--(12.4, 11.8);

\end{tikzpicture}
\caption{The structure of a stair-layered permutation.}
\end{figure}
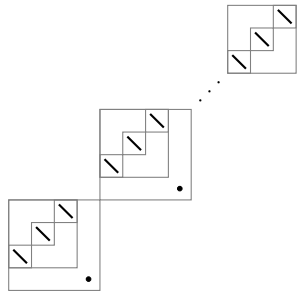

We define a map from $\mathcal{U}_{2k+1}(312, 3421)$ to $\mathcal M_k^{\bf S}$ using similar rules as we did for the previous class. Let $\Lambda$ be the path $EUEU\cdots EU$ containing an alternating $k$ $E$'s and $k$ $U$'s. Consider some $\pi\in\mathcal{U}_{2k+1}(312,3421)$, with $\tl(\pi)=\ell$. Let $a_1, \ldots, a_m$ be the ascent tops of $\pi'=\pi_1\cdots\pi_{2k+1-\ell}$ (which is the permutation obtained from removing the tail from $\pi$), ordered from left to right, and let $n_i$ be the number of descent bottoms to the left of $a_i$. For $1\leq i \leq m$, if $a_i$ is in the same block as the point immediately to the left of it, insert a $D$ immediately after the $n_i^\text{th}$ $U$ in $\Lambda$. If $a_i$ is not in the same block as the point immediately to the left of it, insert a $D$ immediately after the $n_{i}+1^\text{th}$ $E$ in $\Lambda$. Then insert $\ell$ $D$'s at the end of $\Lambda$ to account for the tail. Note that the tail needs to be dealt with in this way because, in this case, the tail can be its own block (in the previous proof, the tail cannot be its own block because we assume blocks to be maximal, and adding a tail to a svee-skew-sum-1 block always creates a svee block). This process is equivalent to, if the tail is its own block, treating the last two blocks as a single block, and then using the rules from the previous proof. Thus, by the same logic in the previous proof, only now with layered permutations instead of svee permutations, this is indeed a bijection between the desired sets, and $|\mathcal{U}_{2k+1}(312, 3421)|=| \mathcal M_k^{\bf S}|=\frac{1}{2k+1}{3k\choose k}$.\end{proof}

\begin{figure}[H]\label{stair-layered bij}
\[\begin{array}{l}
\begin{tikzpicture}[scale=.35]
\draw[fill](1,3) circle [radius=0.2];
\draw[fill] (2,2) circle [radius=0.2];
\draw[fill] (3,4) circle [radius=0.2];
\draw[fill] (4,1) circle [radius=0.2];
\draw[fill] (5,8) circle [radius=0.2];
\draw[fill] (6,7) circle [radius=0.2];
\draw[fill] (7,6) circle [radius=0.2];
\draw[fill] (8,9) circle [radius=0.2];
\draw[fill] (9,11) circle [radius=0.2];
\draw[fill] (10,10) circle [radius=0.2];
\draw[fill] (11,5) circle [radius=0.2];
\draw[fill] (12,12) circle [radius=0.2];
\draw[fill] (13,13) circle [radius=0.2];
\draw [thick] (1,3)--(1,4)--(3,4)--(3,8)--(5,8)--(5,11)--(9,11)--(9,13)--(13,13);
\draw[thick] (6,7)--(6,9)--(8,9);
\draw[thick] (10,10)--(10,12)--(12,12);
\draw[dashed][thick][gray] (0.5,0.5) rectangle (4.5, 4.5);
\draw[dashed][thick][gray] (4.5,4.5) rectangle (11.5, 11.5);
\draw[dashed][thick][red] (11.5,11.5) rectangle (13.5, 13.5);
\end{tikzpicture}\end{array}
\begin{array}{l}
\begin{tikzpicture}[scale=.7]
\draw [thick] (0,0)--(1,0);
\draw [thick] (0.2, -0.2)--(0,0)--(0.2,0.2);
\draw [thick] (0.8, -0.2)--(1,0)--(0.8, 0.2);
\draw [white] (1.1,0)--(1.2,0);
\end{tikzpicture}
\end{array}\begin{array}{l} \begin{tikzpicture}[scale=.45]
\draw[help lines] (0,0) grid (18,2);
\draw [ultra thick][teal] (0,0) -- (1,0)--(2,1)--(3,0)--(4,0)--(5,1)--(6,1)--(7,0)--(8,1)--(9,1)--(10,2)--(12,0)--(13,0)--(14,1)--(15,1)--(16,2)--(18,0);
\end{tikzpicture}\end{array}\] 
\caption{The stair-layered permutation 3\,2\,4\,1\,8\,7\,6\,9\,11\,10\,5\,12\,13 and its image under the bijection in Theorem \ref{stair-layered}, the $\bf{S}$-Motzkin path $\mathit{EUDEUEDUEUDDEUEUDD}$. The rightmost block is the tail.}
\end{figure}
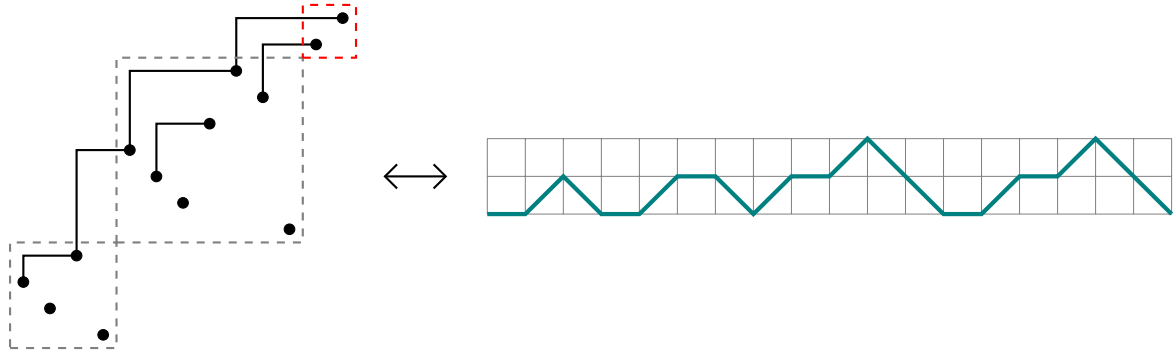

\begin{Th}\label{svee-increasing}
We have $|\mathcal{U}_{2k+1}(312, 1432)|=\frac{1}{2k+1}{3k\choose k}$.
\end{Th}
\begin{proof}

Consider $\pi \in \mathcal{U}_{2k+1}(312, 1432)$ and some consecutive subsequence of $\pi$, called $\rho$. Since $\pi$ avoids 312, we can write $\rho =\lambda\mu\rho_m$, where $\lambda$ is below $\rho_m$ and $\mu$ is above $\rho_m$. Now we distinguish two cases.\\

\noindent\underline{Case 1:} The subsequence $\lambda$ is nonempty. Then an entry of $\lambda$ lies below and to the left of $\mu$ while $\rho_m$ lies above and to the right of $\lambda$ and below and to the right of $\mu$, so in order for $\rho$ to avoid 1432, we have that $\mu$ must avoid 21 and thus is increasing.\\

\noindent\underline{Case 2:} The subsequence $\lambda$ is empty. Then there are no clear restrictions on $\mu$.\\

We first perform the above decomposition on $\pi$. In case 1, we can repeat this decomposition process on $\lambda$; in case 2, we can repeat it on $\mu$. We repeat until the unknown subsequences are empty. The result is the decomposition of $\pi$ into a permutation that looks like a svee permutation, except each point above the horizontal can be replaced with a block of the form $\Inc(n)\ominus 1$ (points that are not replaced with blocks of this form are regarded as blocks of size one). We call a permutation of this form \textit{svee-increasing}.

\begin{figure}[H]
\centering
\begin{tikzpicture}[scale=.3]

\draw[fill] (0,0) circle [radius=0.15];
\draw[fill] (1,-1) circle [radius=0.15];
\draw[fill] (6,-2) circle [radius=0.15];
\draw[fill] (12,-3) circle [radius=0.15];

\draw[gray] (2,1) rectangle (5,4);
\draw [thick] (2.4, 1.9)--(4.2, 3.7);
\draw[fill] (4.5, 1.5) circle [radius=0.15];

\draw[gray] (8,4) rectangle (11,7);
\draw [thick] (8.4, 4.9)--(10.2, 6.7);
\draw[fill] (10.5, 4.5) circle [radius=0.15];

\filldraw[black] (7.5,0) node[anchor=south] {...};

\draw[fill] (13,8) circle [radius=0.15];

\end{tikzpicture}
\caption{The general structure of a svee-increasing permutation. Each of the blocks above the horizontal and the points below the horizontal can be deleted.}
\end{figure}
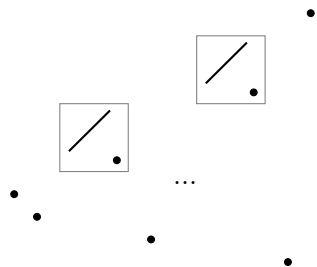

We define a map from $\mathcal{U}_{2k+1}(312, 1432)$ to $\mathcal M_k^{\bf S}$. Given the plot of some $\pi \in \mathcal{U}_{2k+1}(312, 1432)$, we label each point below the horizontal with the symbols $EU$. Above the horizontal, we label each point in its own block $D$, and we label a block of size $n\geq 2$ with $ED^{n-1}U$, where $D^{n-1}$ denotes the concatenation of $n-1$ $D$'s. The first point in $\pi$ is left unlabeled. Then we read the labels from left to right to obtain the path $\Lambda$.

Note that in this labeling, each ascent top can be associated with a $D$ and each descent bottom with an $EU$ (we can think of the $EU$ in $ED^{n-1}U$ being associated with the last point in the block, a descent bottom, and the $D^{n-1}$ being associated with the $n-1$ ascent tops in the block). Thus, since $\pi$ is uniquely sorted, there are $k$ of each type of step. Moreover, by construction, $E$ and $U$ steps always alternate, starting with $E$. Since $\pi$ has a CHC, each prefix of $\pi$ contains at least as many descents as ascents, so each prefix of $\Lambda$, when read right to left, contains at least as many $U$'s as $D$'s and thus remains above the horizontal. Therefore, $\Lambda \in \mathcal M_k^{\bf S}$. 

For the inverse map, Let $\Lambda \in \mathcal M_k^{\bf S}$. Since $\Lambda$ starts with $E$ and between every two east steps there is exactly one up step, we can uniquely factor $\Lambda = B_1 \cdots B_\ell$ where each $B_i$ is either $D$ or of the form $ED^mU$ for some $m \geq 0$. To recover $\pi \in \mathcal{U}_{2k+1}(312, 1432)$, we start with a permutation plot with a single point. Then we read $\Lambda$ left to right and for each $EU$ factor we plot a point below and to the right of the plot so far, for each $D$ factor we plot a point above and to the right of the plot so far, and for each $ED^mU$ factor ($m \geq 1)$ we plot a block of the form Inc$(m) \ominus 1$ above and to the right of the plot so far. The resulting $\pi$ is svee-increasing and thus avoids 321 and 1432. By the same logic as in the paragraph above, in this process $D$ steps correspond to ascent tops and $EU$ pairs correspond to descent bottoms, so the condition that the $i^\text{th}$ down step in $\Lambda$ occurs after at least $i$ east steps and $i$ up steps implies that $\pi$ is uniquely sorted. Thus, $\pi \in \mathcal{U}_{2k+1}(312, 1432)$.
Therefore, this is indeed a bijection, so $|\mathcal{U}_{2k+1}(312, 1432)|=| \mathcal M_k^{\bf S}|=\frac{1}{2k+1}{3k\choose k}$.\end{proof}

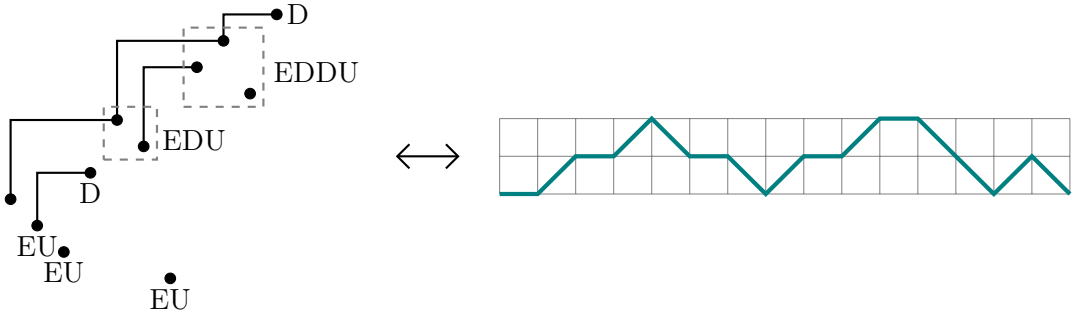
\begin{figure}[H]\label{svee-increasing bij}
\[\begin{array}{l}
\begin{tikzpicture}[scale=.35]
\draw[fill](1,4) circle [radius=0.2];
\draw[fill] (2,3) circle [radius=0.2];
\draw[fill] (3,2) circle [radius=0.2];
\draw[fill] (4,5) circle [radius=0.2];
\draw[fill] (5,7) circle [radius=0.2];
\draw[fill] (6,6) circle [radius=0.2];
\draw[fill] (7,1) circle [radius=0.2];
\draw[fill] (8,9) circle [radius=0.2];
\draw[fill] (9,10) circle [radius=0.2];
\draw[fill] (10,8) circle [radius=0.2];
\draw[fill] (11,11) circle [radius=0.2];
\draw [thick] (1,4)--(1,7)--(5,7)--(5,10)--(9,10)--(9,11)--(11,11);
\draw[thick] (2,3)--(2,5)--(4,5);
\draw[thick] (6,6)--(6,9)--(8,9);
\draw[dashed][thick][gray] (4.5,5.5) rectangle (6.5, 7.5);
\draw[dashed][thick][gray] (7.5,7.5) rectangle (10.5, 10.5);
\node[below] at (2,3) {EU};
\node[below] at (3,2) {EU};
\node[below] at (7,1) {EU};
\node[below] at (4,5) {D};
\node[below] at (12.5,9.6) {EDDU};
\node[right] at (11,11) {D};
\node[below] at (7.9,7) {EDU};
\end{tikzpicture}\end{array}
\begin{array}{l}
\begin{tikzpicture}[scale=.7]
\draw [thick] (0,0)--(1,0);
\draw [thick] (0.2, -0.2)--(0,0)--(0.2,0.2);
\draw [thick] (0.8, -0.2)--(1,0)--(0.8, 0.2);
\draw [white] (1.1,0)--(1.2,0);
\end{tikzpicture}
\end{array}\begin{array}{l} \begin{tikzpicture}[scale=.45]
\draw[help lines] (0,0) grid (15,2);
\draw [ultra thick][teal] (0,0) -- (1,0)--(2,1)--(3,1)--(4,2)--(5,1)--(6,1)--(7,0)--(8,1)--(9,1)--(10,2)--(11,2)--(13,0)--(14,1)--(15,0);
\end{tikzpicture}\end{array}\] 
\caption{The svee-increasing permutation 4\,3\,2\,5\,7\,6\,1\,9\,10\,8\,11 and its image under the bijection in Theorem \ref{svee-increasing}, the $\bf{S}$-Motzkin path $\mathit{EUEUDEDUEUEDDUD}$.}
\end{figure}

\begin{Th}\label{vee-layered}
We have $|\mathcal{U}_{2k+1}(231, 1423)|=\frac{1}{2k+1}{3k\choose k}$.
\end{Th}
\begin{proof}

Consider $\pi \in \mathcal{U}_{2k+1}(231, 1423)$ and some consecutive subsequence of $\pi$, called $\rho$. Since $\rho$ avoids 231, we can write $\rho = \rho_1\lambda \mu$, where $\lambda$ lies below $\rho_1$ and $\mu$ lies above $\rho_1$. Since $\rho$ avoids 1423 and $\rho_1$ is below and to the left of  $\mu$, we have that  $\mu$ must also avoid 312 in addition to 231, and thus $\mu$ is layered. 

We first perform the above decomposition on $\pi$. Once we find layered $\mu$, we can repeat this same decomposition process on $\lambda$ instead of $\pi$ and continue to repeat until the unknown subsequence is empty. Since at each step we add a point to the the left and above the unknown part of the permutation, as well as a layered permutation to the right of that point and the unknown part and above the unknown part, the resulting permutation is vee shaped, except each point to the right of the vertical of the vee can be replaced with a decreasing block. We call a permutation of this form \textit{vee-layered}.

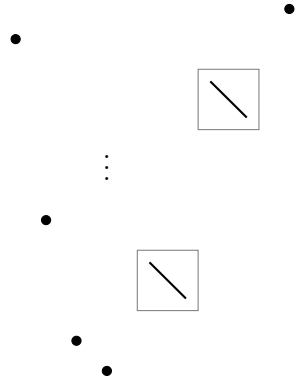
\begin{figure}[H]
\centering
\begin{tikzpicture}[scale=.4]

\draw[fill] (0,0) circle [radius=0.15];
\draw[fill] (-1,1) circle [radius=0.15];
\draw[fill] (-2,5) circle [radius=0.15];
\draw[fill] (-3,11) circle [radius=0.15];
\draw[fill] (6,12) circle [radius=0.15];

\draw[gray] (1,2) rectangle (3,4);
\draw [thick] (1.4, 3.6)--(2.6, 2.4);

\filldraw[black] (0,6) node[anchor=south] {\vdots};

\draw[gray] (3,8) rectangle (5,10);
\draw [thick] (3.4, 9.6)--(4.6, 8.4);

\end{tikzpicture}
\caption{The general structure of a vee-layered permutation. Each of the blocks to the right of the vertical and the points to the left of the vertical can be deleted.}
\end{figure}

We define a map from $\mathcal{U}_{2k+1}(231, 1423)$ to $\mathcal M_k^{\bf S}$ using rules similar to those in the previous bijection: Given the plot of $\pi \in \mathcal{U}_{2k+1}(231, 1423)$, we label each point to the left of the vertical with $D$. To the right of the vertical, we label each layer of size $n$ with $ED^{n-1}U$. The lowest point in $\pi$ is left unlabeled. Then we read the labels from top to bottom to obtain the path $\Lambda$. By the same logic as in the previous proof, except where now the separation of $E$ and $U$ with $D$'s marks a layer with one ascent top and $n-1$ descent bottoms, we have that $\Lambda$ is an $\bf{S}$-Motzkin path. For the inverse map, we can again factor $\Lambda = B_1 \cdots B_\ell$ where each $B_i$ is of the form $ED^mU$ ($m \geq 0$) or $D$. Then, we start with the plot of a single point and in the reverse order $B_\ell, \ldots, B_1$, we plot a point above and to the left of the plot so far for each $D$ factor, and plot a decreasing block of size $m+1$ above and to the right of the plot so far for each $ED^mU$ factor. By the same logic as before, the resulting $\pi$ is in $ \mathcal{U}_{2k+1}(231, 1423)$. Thus, this is indeed a bijection, so $|\mathcal{U}_{2k+1}(231, 1423)|=| \mathcal M_k^{\bf S}|=\frac{1}{2k+1}{3k\choose k}$.
\end{proof}

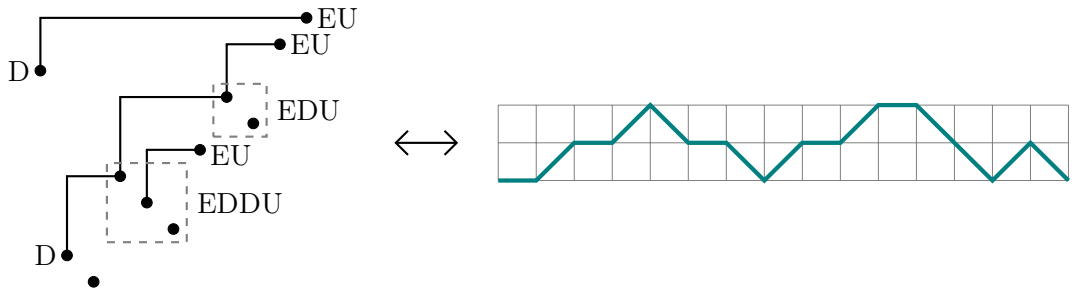
\begin{figure}[H]\label{vee-layered bij}
\[\begin{array}{l}
\begin{tikzpicture}[scale=.35]
\draw[fill](1,9) circle [radius=0.2];
\draw[fill] (2,2) circle [radius=0.2];
\draw[fill] (3,1) circle [radius=0.2];
\draw[fill] (4,5) circle [radius=0.2];
\draw[fill] (5,4) circle [radius=0.2];
\draw[fill] (6,3) circle [radius=0.2];
\draw[fill] (7,6) circle [radius=0.2];
\draw[fill] (8,8) circle [radius=0.2];
\draw[fill] (9,7) circle [radius=0.2];
\draw[fill] (10,10) circle [radius=0.2];
\draw[fill] (11,11) circle [radius=0.2];
\draw [thick] (1,9)--(1,11)--(11,11);
\draw[thick] (2,2)--(2,5)--(4,5)--(4,8)--(8,8)--(8,10)--(10,10);
\draw[thick] (5,4)--(5,6)--(7,6);
\draw[dashed][thick][gray] (3.5,2.5) rectangle (6.5, 5.5);
\draw[dashed][thick][gray] (7.5,6.5) rectangle (9.5, 8.5);
\node[left] at (1,9) {D};
\node[left] at (2,2) {D};
\node[right] at (11,11) {EU};
\node[right] at (10,10) {EU};
\node[right] at (7,5.8) {EU};
\node[right] at (9.5,7.5) {EDU};
\node[right] at (6.5,4) {EDDU};
\end{tikzpicture}\end{array}
\begin{array}{l}
\begin{tikzpicture}[scale=.7]
\draw [thick] (0,0)--(1,0);
\draw [thick] (0.2, -0.2)--(0,0)--(0.2,0.2);
\draw [thick] (0.8, -0.2)--(1,0)--(0.8, 0.2);
\draw [white] (1.1,0)--(1.2,0);
\end{tikzpicture}
\end{array}\begin{array}{l} \begin{tikzpicture}[scale=.45]
\draw[help lines] (0,0) grid (15,2);
\draw [ultra thick][teal] (0,0) -- (1,0)--(2,1)--(3,1)--(4,2)--(5,1)--(6,1)--(7,0)--(8,1)--(9,1)--(10,2)--(11,2)--(13,0)--(14,1)--(15,0);
\end{tikzpicture}\end{array}\] 
\caption{The vee-layered permutation 9\,2\,1\,5\,4\,3\,6\,8\,7\,10\,11 and its image under the bijection in Theorem \ref{vee-layered}, the $\bf{S}$-Motzkin path $\mathit{EUEUDEDUEUEDDUD}$.}
\end{figure}

\begin{Th}
We have $|\mathcal{U}_{2k+1}(132, 3412)|=\frac{1}{2k+1}{3k\choose k}$.
\end{Th}
\begin{proof}

The result follows from Theorem \ref{vee-layered} and some previous work. Theorem 6.1 in \cite{main} states that there exists a bijection $\text{swu}: \mathcal{U}_{2k+1}(231)\to\mathcal{U}_{2k+1}(132)$. In the proof of Theorem 5.1 in \cite{fertility}, Defant shows that $\text{swu}(\text{Av}(231, 1423)) = 
\text{Av}(132, 3412)$. Consequently, swu yields a bijection between $\mathcal U_{2k+1}(231,1423)$ and $\mathcal U_{2k+1}(132,3412)$. Thus $|\mathcal U_{2k+1}(132,3412)|=|\mathcal U_{2k+1}(231,1423)|=\frac{1}{2k+1}{3k\choose k}$. \end{proof}

\section{A Bijection with a Subclass of Schr\"oder Paths}\label{schroder}

We now introduce another well-studied type of lattice path. A \textit{Schr\"oder path} of semilength $k$ is a path from $(0,0)$ to $(2k, 0)$ that consists of an equal number of up steps and down steps, as well as some number of (2,0) steps (called \textit{horizontal steps} and denoted $H$ in the up-down sequence) and at no point crosses below the horizontal axis. The $k^\text{th}$ \textit{Schr\"oder number}, denoted $\mathscr{S}_k$, is defined to be the number of Schr\"oder paths of semilength $k$.

\begin{figure}[H]
\centering
\begin{tikzpicture}[scale=.5]
\draw[help lines] (0,0) grid (14,4);
\draw [ultra thick][teal] (0,0)--(1,1)--(3,1)--(6,4)--(8,2)--(9,3)--(11,3)--(14,0);
\end{tikzpicture}
\caption{The Schr\"oder path $\mathit{UHUUUDDUHDDD}$ of semilength 7.}
\end{figure}
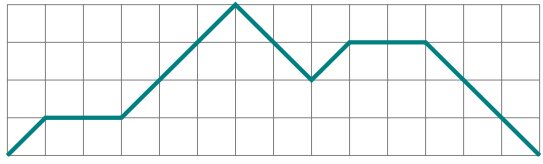

It is well-known that the Schr\"oder paths can be divided into two equinumerous subclasses: paths that have a horizontal step on the horizontal axis and paths that \textit{do not} have a horizontal step on the horizontal axis. Let $\mathcal{S}_k$ be the set of paths in the latter subclass of length $k$. It follows that $|\mathcal{S}_k|=\mathscr{S}_k/2$ for $k\geq 1$ (and $|\mathcal{S}_0|=\mathscr{S}_0=1$); these numbers are known as the \emph{little Schr\"oder numbers}\footnote{Sometimes they are called the \textit{small Schr\"oder numbers}.}.

\begin{Th}\label{sch}
We have $|\mathcal{U}_{2k+1}(231,1432)|=|\mathcal{S}_k|$.
\end{Th}

\begin{proof}
Consider $\pi \in \mathcal{U}_{2k+1}(231, 1432)$ and some consecutive subsequence of $\pi$, called $\rho$. Since $\rho$ avoids 231, we can write $\rho=\lambda\rho_m\mu$, where $\rho_m$ is the largest entry in $\rho$ and $\lambda$ is below $\mu$. We now consider two cases.\\

\noindent\underline{Case 1}: The subsequence $\lambda$ is nonempty. Then, since a point in $\lambda$ lies below and to the left of $\mu$ and $\rho_m$ lies above and horizontally between $\lambda$ and $\mu$, we have that $\mu$ must not contain 21 in order for $\rho$ to avoid 1432. Thus, $\mu$ is increasing. If $\mu$ has size $n\geq 2$, then $\rho_m$ and $\mu_n$ are both ascent tops and thus NE endpoints of hooks $H_i$ and $H_j$, respectively, in $\pi$'s CHC. If $\pi_i, \pi_j$ are the SW endpoints of $H_i, H_j$, respectively, then $i<j$ in order for the hooks to not intersect. But then $\rho_m$ would be the NE endpoint of $H_j$, and not $\mu_n$, because $\rho_m$ is to the left of and immediately above $\mu_n$. Thus, $\mu$ must have size at most 1, giving us that $\rho=\lambda\oplus \tau$, where $\tau$ is $1$ or $21$.\\

\noindent\underline{Case 2}: The subsequence $\lambda$ is empty. Then there are no clear restrictions on $\mu$.\\

We first perform this decomposition process on $\pi$. In case 1, we can repeat this decomposition process on $\lambda$; in case 2, we can repeat it on $\mu$. We repeat until the unknown subsequences are empty. The result is the decomposition $\pi=L1R$ where $L$ is decreasing and $R$'s normalization is $\tau_1\oplus \cdots \oplus \tau_\ell$ where each $\tau_i$ is $1$ or $21$. Thus $\pi$ is a permutation that is vee, except points to the right of the vertical can be replaced with the block 21. We call this type of permutation \textit{vee-step}. Note that, by definition, the set of vee-step permutations is the set of vee-layered permutations (described in the proof of Theorem \ref{vee-layered}) in which each layer is at most size 2.

Now we define a map from $\mathcal{U}_{2k+1}(231,1432)$ to $\mathcal{S}_k$ as follows: Given $\pi \in \mathcal{U}_{2k+1}(231,1432)$, we write $\pi = L1R$ as before and label every point in $L$ with the symbol $U$. Recall $R$ has normalization $\tau_1\oplus \cdots \oplus \tau_\ell$ where $\tau_i = 1$ or $21$. Treating each $\tau_i$ as a single unit, if $\tau_i=1$ we label the corresponding point in $R$ with a $D$ and if $\tau_i=21$ we label the corresponding pair points in $R$ with a single $H$. Then, starting at the lowest point 1 in $\pi$, we read off the labels from bottom to top to obtain a path $\Lambda$. Since $\pi$ is uniquely sorted, it has $k$ ascents and $k$ descents. The steps labeled with $H$ contain both an ascent top and a descent bottom, so $\Lambda$ has an equal number of $U$'s and $D$'s. Moreover, $\pi$ being sorted implies not only that every prefix of $\Lambda$ contains as many $U$'s as $D$'s, but also that there are no horizontal steps in $\Lambda$ along the horizontal axis. This is because $H$ corresponds to a 21 block, which contains an ascent top followed by a descent bottom. If this occurred along the horizontal axis, there would be a prefix subsequence $\pi_1\cdots\pi_m$ with more ascent tops than descent bottoms, meaning that there must be an ascent top that is not a NE endpoint, which contradicts Lemma \ref{partition}. Thus, $\Lambda\in \mathcal{S}_k$. The map is easily reversible, making it a bijection. Therefore $|\mathcal{U}_{2k+1}(231,1432)|=|\mathcal{S}_k|$.\end{proof}

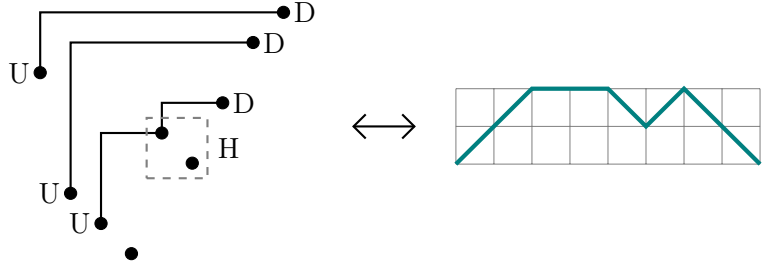
\begin{figure}[H]\label{shro bij}
\[\begin{array}{l}
\begin{tikzpicture}[scale=.4]
\draw[fill] (1,7) circle [radius=0.2];
\draw[fill] (2,3) circle [radius=0.2];
\draw[fill] (3,2) circle [radius=0.2];
\draw[fill] (4,1) circle [radius=0.2];
\draw[fill] (5,5) circle [radius=0.2];
\draw[fill] (6,4) circle [radius=0.2];
\draw[fill] (7,6) circle [radius=0.2];
\draw[fill] (8,8) circle [radius=0.2];
\draw[fill] (9,9) circle [radius=0.2];
\draw [thick] (1,7)--(1,9)--(9,9);
\draw [thick] (2,3)--(2,8)--(8,8);
\draw [thick] (3,2)--(3,5)--(5,5)--(5,6)--(7,6);
\draw [dashed, thick, gray] (4.5,3.5) rectangle (6.5, 5.5);
\node [right] at (9,9) {D};
\node [right] at (8,8) {D};
\node [right] at (7,6) {D};
\node [right] at (6.5,4.5) {H};
\node [left] at (1,7) {U};
\node [left] at (2,3) {U};
\node [left] at (3,2) {U};
\end{tikzpicture}\end{array}
\begin{array}{l}
\begin{tikzpicture}[scale=.7]
\draw [thick] (0,0)--(1,0);
\draw [thick] (0.2, -0.2)--(0,0)--(0.2,0.2);
\draw [thick] (0.8, -0.2)--(1,0)--(0.8, 0.2);
\draw [white] (1.1,0)--(1.2,0);
\end{tikzpicture}
\end{array}\begin{array}{l} \begin{tikzpicture}[scale=.45]
\draw[help lines] (0,0) grid (8,2);
\draw [ultra thick][teal] (0,0) -- (2,2) -- (4,2) -- (5,1) -- (6,2) -- (8,0);
\end{tikzpicture}\end{array}\] 
\caption{The vee-step permutation 732154689 and its image under the bijection in Theorem \ref{sch}, the Schr\"oder path $\mathit{UUHDUDD}$.}
\end{figure}

\section{Counting $\mathcal{U}_{2k+1}(231, 4312)$ with a Generating Function}

\cite{main} showed how to decompose permutations in $\mathcal{U}_{2k+1}(231, 4132)$ to obtain an identity proving these permutations are in bijection with Pallo comb intervals. In this section, we modify Defant's method to instead decompose the permutations in $\mathcal{U}_{2k+1}(231, 4312)$, which are counted by the same sequence.

\begin{Th}
We have
    \begin{center}
$\sum_{k\geq 0}|\mathcal{U}_{2k+1}(231, 4312)|x^k=C(xC(x))$,
    \end{center}
\end{Th}

\begin{center}
    \textit{where $C(x)=\frac{1-\sqrt{1-4x}}{2x}$ is the generating function of the sequence of Catalan numbers}.
\end{center}

\begin{proof}
Recall that a uniquely sorted permutation $\pi=\pi_1\cdots\pi_{2k+1}$ is called \emph{nice} if the point $(1,\pi_1)$ is the southwest endpoint of a hook whose northeast endpoint is $(2k+1,2k+1)$. Consider $\pi \in \mathcal{U}_{2k+1}(231, 4312)$ and, for now, assume that $\pi$ is nice. Since $\pi$ avoids 231, we are able to write $\pi=\pi_1\lambda\mu(2k+1)$, where $\lambda\in S_{\pi_1-1}$(which is nonempty because, by assumption, $1$ is a descent) and $\mu$ is a permutation of $\pi_1+1,\ldots, 2k$. Since $\lambda$ is a subsequence of $\pi$, it avoids 231. Since $\pi$ avoids 4312, $\lambda$ must avoid 312, so $\lambda$ is layered. 

Note that the first layer in $\lambda$ cannot be a single point. If it were, $\lambda_2$ would be an ascent top and thus a NE endpoint of a hook. But the only descent top before $\lambda_2$ is $\pi_1$, which lies above $\lambda_2$, a contradiction.

Let $m$ be the largest integer such that the normalization of $\lambda_1\cdots \lambda_{2m+1}$ is in $\mathcal{U}_{2m+1}(231, 312)$ and such that $\lambda_{2m+1}$ is the first entry in a layer of size at least two. We let $\tau=\lambda_1\cdots\lambda_{2m+1}$. Since the first layer of $\lambda$ has size at least two, and the single entry $\lambda_1$ is in $\mathcal{U}_{1}(231, 312)$, such a $\tau$ always exists. Let $\sigma$ be the remaining entries of $\lambda$, so that $\tau\sigma=\lambda$.

Now, let $\pi'=\pi_1\sigma\mu$, that is, the permutation obtained by removing $\tau$ and $(2k+1)$ from $\pi$. We claim that $\pi' \in \mathcal{U}_{2k-2m-1}(231, 4312)$. Because $\pi'$ is a subsequence of $\pi$, it avoids 231 and 4312. Since to obtain $\pi'$ we remove one permutation of length $2m+1$ and one of length 1, we have that $\pi'$ has length $2k-2m-1$.

Since $\pi\in \mathcal{U}_{2k+1}$ and $\tau \in \mathcal{U}_{2m+1}$, the permutations $\pi$ and $\tau$ have $k$ descents and $m$ descents, respectively. Thus when we take out $\tau$, $\pi'$ loses $m$ descents. Because all of $\sigma$ lies below $\pi_1$, we have that $\pi_1$ remains a descent top in $\pi'$. Since $(2k+1, 2k+1)$ has the maximal height in the plot of $\pi$, the point $(2k, \pi_{2k})$ was not a descent top of $\pi$, so when we remove $2k+1$, $\pi'$ does not lose any descents. Originally $\lambda_{2m+1}$ was not the last entry in a decreasing block, so it was a descent top in $\pi$ but now is not in $\tau$. Thus $\pi'$ has $k-m-1 = \frac{2k-2m-1-1}{2}$ descents.

All that is left to check is that $\pi'$ has a CHC. Note that $\tau$, when considered its own permutation, is uniquely sorted, and appears all the way on the left of $\pi$, save for $\pi_1$. Thus in $\pi$ all hooks with an endpoint in $\tau$ have the other endpoint in $\tau$, except for the last point in $\tau$, which is a descent top in $\pi$ and thus is the SW endpoint of some hook $H_i$. Then when we remove $\tau$ and $(2k+1, 2k+1)$ to obtain $\pi'$, all hooks in $\pi'$ are preserved except $H_1$, which is missing a NE endpoint, and $H_i$, which is missing a SW endpoint. We create the hook $H_1'$ connecting $\pi'_1$ to the NE endpoint of $H_i$, which resolves the issue and results in a CHC of $\pi'$. Thus, $\pi' \in \mathcal{U}_{2k-2m-1}(231, 4312)$. Now, let $\pi''$ be the normalization of $\pi'$ and $\tau'$ be the normalization of $\tau$. We have obtained $(\pi'',\tau') \in \mathcal{U}_{2k-2m-1}(231, 4312) \times \mathcal{U}_{2m+1}(231, 312)$ from the original $\pi$.

We now show this process is reversible. Given $(\pi'', \tau') \in \mathcal{U}_{2k-2m-1}(231, 4312) \times \mathcal{U}_{2m+1}(231, 312)$, we can insert the plot of $\tau'$ under $\pi''_1$ and merge its last block with the first block in $\pi''_2\cdots\pi''_{2k-2m-1}$. We then append $(2k+1)$ to recover $\pi$. Then, for the CHC, we create hook $H_1$ between $\pi_1$ and $\pi_{2k+1}$ and create a new hook between the NE endpoint of $\pi''_1$'s hook and the new descent top created by merging $\tau$. The remaining hooks are unaffected by the merge. Thus, this decomposition is bijective, so nice permutations in $\mathcal{U}_{2k+1}(231, 4312)$ correspond to pairs of permutations in $\mathcal{U}_{2k-2m-1}(231, 4312) \times \mathcal{U}_{2m+1}(231, 312)$.

From Lemma \ref{layered}, we have that $|\mathcal{U}_{2m+1}(231, 312)|=C_m$, so $\sum_{m\geq 0}|\mathcal{U}_{2m+1}(231, 312)|x^m =C(x)$. Reindexing gives $\sum_{n\geq 1}|\mathcal{U}_{n}(231, 312)|=xC(x^2)$. Let $B(x) = \sum_{k\geq 0}|\mathcal{U}_{2k+1}(231, 4312)|x^k$ and $\widetilde{B}(x)= \sum_{n\geq 1}|\mathcal{U}_{n}(231, 4312)|x^n$. Then nice permutations in $\mathcal{U}_{2k+1}(231, 4312)$ are counted by the generating function $[\widetilde{B}(x)][xC(x^2)][x]= x^2C(x^2)\Tilde{B}(x)$. The rest of the proof is identical to that of Theorem 8.1 in \cite{main}: using the generating function for nice permutations in $\mathcal{U}_{2k+1}(231, 4312)$ to count all such permutations, we can prove that $\tilde{B}(x)=x+xC(x^2)\tilde{B}(x)^2$, from which it follows that $B(x)=C(xC(x))$, as desired. Refer to Defant's proof for details.
\end{proof}

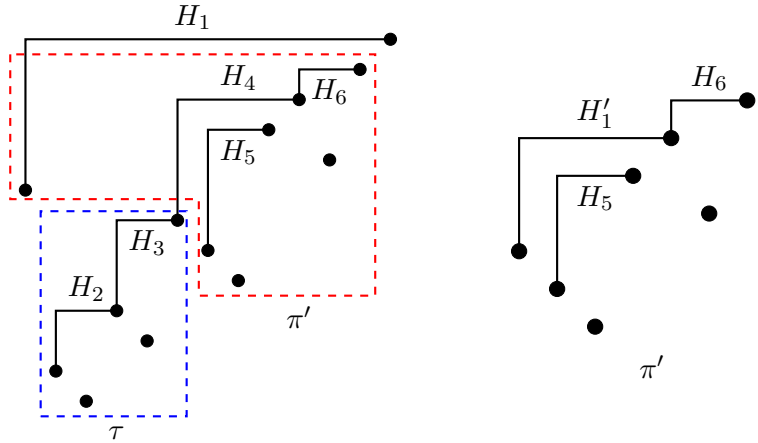
\begin{figure}[H]\label{generating}
\[\begin{array}{l}
\begin{tikzpicture}[scale=.4]
\draw[fill] (1,8) circle [radius=0.2];
\draw[fill] (2,2) circle [radius=0.2];
\draw[fill] (3,1) circle [radius=0.2];
\draw[fill] (4,4) circle [radius=0.2];
\draw[fill] (5,3) circle [radius=0.2];
\draw[fill] (6,7) circle [radius=0.2];
\draw[fill] (7,6) circle [radius=0.2];
\draw[fill] (8,5) circle [radius=0.2];
\draw[fill] (9,10) circle [radius=0.2];
\draw[fill] (10,11) circle [radius=0.2];
\draw[fill] (11,9) circle [radius=0.2];
\draw[fill] (12,12) circle [radius=0.2];
\draw[fill] (13,13) circle [radius=0.2];
\draw [thick] (2,2)--(2,4)--(4,4)--(4,7)--(6,7)--(6,11)--(10,11)--(10,12)--(12,12);
\draw [thick] (1,8)--(1,13)--(13,13);
\draw [thick] (7,6)--(7,10)--(9,10);
\draw [dashed, thick][blue] (1.5,0.5) rectangle (6.3,7.3);
\draw [dashed, thick][red] (0.5,12.5) -- (0.5, 7.7) -- (6.7,7.7)--(6.7,4.5)--(12.5, 4.5)--(12.5,12.5)--(0.5,12.5);

\node [above] at (6.5, 13) {$H_1$};
\node [above] at (3, 4) {$H_2$};
\node [below] at (5, 7) {$H_3$};
\node [above] at (8, 11) {$H_4$};
\node [below] at (8, 10) {$H_5$};
\node [below] at (11, 12) {$H_6$};
\node [below] at (4, 0.5) {$\tau$};
\node [below] at (10, 4.5) {$\pi'$};

\end{tikzpicture}\end{array}\quad\quad\quad
\begin{array}{l} \begin{tikzpicture}[scale=.5]
\draw[fill] (1,3) circle [radius=0.2];
\draw[fill] (2,2) circle [radius=0.2];
\draw[fill] (3,1) circle [radius=0.2];
\draw[fill] (4,5) circle [radius=0.2];
\draw[fill] (5,6) circle [radius=0.2];
\draw[fill] (6,4) circle [radius=0.2];
\draw[fill] (7,7) circle [radius=0.2];
\draw [thick] (1,3)--(1,6)--(5,6)--(5,7)--(7,7);
\draw [thick] (2,2)--(2,5)--(4,5);
\node [below] at (4.5, 0.5) {$\pi'$};

\node [above] at (3, 6) {$H_1'$};
\node [above] at (6, 7) {$H_6$};
\node [below] at (3, 5) {$H_5$};
\end{tikzpicture}\end{array}\] 
\caption{Decomposing the nice permutation $\pi$ into $\tau$ and $\pi'$.}
\end{figure}

\section{Conclusion}
The theorems in this paper prove nine out of the eighteen conjectures in \cite{main}, which enumerate classes of the form $\mathcal{U}_{2k+1}(\tau^{(1)}, \tau^{(2)})$, where $\tau^{(1)}$ has length 3 and $\tau^{(2)}$ has length 4. The nine remaining conjectures are given again in Table \ref{Tab2}.

{\tabulinesep=1.2mm
\begin{table}[h]
\centering

\begin{tabu}{|c|[1.5pt]c|}

    \hline Patterns & OEIS Sequence \\\tabucline[2.5pt]{-}
    $312,1243$ &  A122368 \\\hline
    $231,1243$ & A001700  \\\hline \noalign{\vskip -.13cm}
    \makecell{\,\,$132,2341$ \vspace{.05cm}\\ $132,4123$} & \hspace{.13cm}A109081 \vspace{-.155cm} \\\hline
    $312,2341$ & A006605 \\\hline

\end{tabu}
\begin{tabu}{|c|[1.5pt]c|}

    \hline Patterns & OEIS Sequence \\\tabucline[2.5pt]{-}
    $312,3241$ & A279569 \\\hline
    $312,4321$ & A063020 \\\hline
    $132,4231$ & A071725 \\\hline
    $231,4321$ & A056010 \\\hline
    \quad & \quad \\\hline
    
\end{tabu}\vspace{.3cm}
\caption{The remaining conjectural OEIS sequences enumerating sets of the form $\mathcal U_{2k+1}(\tau^{(1)},\tau^{(2)})$.}
\label{Tab2}
\end{table}}

It should be noted that the remaining conjectures include $\mathcal{U}_{2k+1}(231, 1243)$ being counted by ${2k-1\choose k}$, which is the same sequence counting the two classes in Section \ref{almost}. Decomposition gives that for a permutation $\pi$ in this class, we can write  $\pi= DD'\tau I$ where $D,D'$ are decreasing, $I$ is increasing, and $\tau$ is vee and is below $D$ and $I$ but is above $D'$. However, the author was unable to find a way to count these permutations and encourages the reader to try. 

The other eight remaining classes are more difficult because, while some are counted by special lattice paths according to the OEIS, these paths seem to either not have the correct length or not have the correct properties to create bijections like the ones in this paper. For example, the elements of $\mathcal{U}_{2k+1}(132, 2341)$ and $\mathcal{U}_{2k+1}(132, 4123)$ are counted by Motzkin paths of length $2k-3$ with no downsteps in even positions. But general Motzkin paths do not require a specific number of down and up steps and $2k-3$ is less than $2k+1$, so it would be difficult to find a bijection that associates descents and ascents in these permutations with certain patterns of steps. That said, these remaining classes could be enumerated through other methods such as the direct counting that we do in Theorem \ref{modveethm}, or, more likely, generating functions. Another possible route would be defining new types of lattice paths that are in bijection with these classes and then enumerating the new paths, which should be easier since there is a greater body of work on the properties of and counting of lattice paths.

Also, according to the data generated in \cite{main}, the 24 sequences counting classes of the form $\mathcal{U}_{2k+1}(\tau)$ where $\tau$ has length four appear to be new and thus studying these classes is likely to be more challenging than the classes studied in this paper. According to the same author, the sequence counting $\mathcal{U}_{2k+1}(231,4123)$ is also not in the OEIS.

Beyond this, the natural next step is to enumerate classes of the form $\mathcal U(\tau^{(1)},\tau^{(2)})$, where $\tau^{(1)}$ and $\tau^{(2)}$ both have length 4, for there is previous work devoted to counting general permutations avoiding two patterns of length four that could prove to be very useful; see \cite{c1, kitaev, linton}. Moreover, some of these classes have nice descriptions, such as the class of skew-merged permutations (which avoid 3412 and 2143) and the class of separable permutations (which avoid 2413 and 3142).

\acknowledgments{The author would like to thank Professor Joe Gallian and advisors Colin Defant and Aaron Berger for organizing the Duluth Research Experience for Undergraduates, the lovely and supportive program where this research took place. The author would also like to thank Defant for introducing this problem and being a constant resource throughout the process of solving it. Lastly, a special thanks to Defant, Gallian, Berger, and former Duluthian Marisa Gaetz for providing invaluable feedback on earlier iterations of the present paper.}

\bibliographystyle{abbrvnat}
\bibliography{bibliography}

\end{document}